\documentclass[12pt]{amsart}
\usepackage[utf8]{inputenc}
\usepackage{tikz} 
\usepackage{tikz-cd} 
\usepackage{tikzit} %

\tikzstyle{red}=[fill=red, draw=black, shape=circle]
\tikzstyle{green}=[fill=white, draw=white, shape=rectangle, tikzit draw=black]

\tikzstyle{arrow}=[->, yshift=40pt]
\tikzstyle{stealth}=[->, >=stealth]
\tikzstyle{dashouille}=[->, dashed]

\usepackage{amsthm} 
\usepackage{amsmath} 
\usepackage{amssymb}
\usepackage{mathtools} 
\usepackage{hyperref}
\usepackage[capitalise,noabbrev]{cleveref} 
\usepackage{geometry}
\usepackage{xcolor}
\usepackage{comment} 

\newcommand{\N}{\mathsf{N}}
\newcommand{\T}{\mathsf{T}}
\newcommand{\E}{\mathsf{E}}
\newcommand{\C}{\mathsf{C}}

\newcommand{\Q}{\mathsf{Q}}
\newcommand{\LN}{\mathsf{LN}}
\newcommand{\LA}{\mathsf{LA}}
\newcommand{\LX}{\mathsf{LX}}
\newcommand{\A}{\mathsf{A}}
\newcommand{\B}{\mathsf{B}}

\newcommand{\LT}{\mathsf{LT}}

\newcommand{\LQ}{\mathsf{LQ}}
\newcommand{\X}{\mathsf{X}}
\newcommand{\Y}{\mathsf{Y}}
\newcommand{\W}{\mathsf{W}}
\newcommand{\M}{\mathsf{M}}
\renewcommand{\L}{\mathsf{L}}
\renewcommand{\P}{\mathsf{P}}

\renewcommand{\ker}{\mathsf{ker}}
\newcommand{\PAN}{\mathsf{P_AN}}
\newcommand{\PAQ}{\mathsf{P_AQ}}
\newcommand{\PAT}{\mathsf{P_AT}}
\newcommand{\K}{\mathsf{K}}

\newcommand{\I}{\mathsf{I}}
\title{Admissibility of Localizations of Crossed Modules}
\author{Olivia Monjon}
\author{J\'er\^ome Scherer}
\address{Mathematics, Ecole Polytechnique F\'ed\'erale de Lausanne, EPFL, Switzerland}
\email{olivia.monjon@gmail.com}
\address{Mathematics, Ecole Polytechnique F\'ed\'erale de Lausanne, EPFL, Switzerland}
\email{jerome.scherer@epfl.ch}
\author{Florence Sterck}
\address{Institut de Recherche en Mathématique et Physique, Universit\'e catholique de Louvain, Belgium}
\email{ florence.sterck@uclouvain.be}



\newtheorem{theorem}{Theorem}[section]
\newtheorem{lemma}[theorem]{Lemma}
\newtheorem{prop}[theorem]{Proposition}
\newtheorem{corollary}[theorem]{Corollary}

\theoremstyle{definition}
\newtheorem{definition}[theorem]{Definition}

\newtheorem{example}[theorem]{Example}
\newtheorem{remark}[theorem]{Remark}

\usepackage{biblatex}
\begin{document}
\begin{abstract} 
The correspondence between the concept of conditional flatness and admissibility in the sense of Galois appears in the context of localization functors in any semi-abelian category admitting a fiberwise localization. It is then natural to wonder what happens in the category of crossed modules where fiberwise localization is not always available. In this article, we establish an equivalence between conditional flatness and admissibility in the sense of Galois (for the class of regular epimorphisms) for regular-epi localization functors. We use this equivalence to prove that nullification functors are admissible for the class of regular epimorphisms, even if the kernels of their localization morphisms are not acyclic.

\end{abstract}
\maketitle

\section*{Introduction} 
 It is a natural question to ask whether the pullback of a nice extension inherits these nice properties. When working with localization functors or reflections one particularly nice feature for an extension is flatness. We say that an extension is $\L$-flat, for a localization functor $\L$, if applying $\L$ to the extension yields another extension, see \cref{def:flat}. The question is thus to understand when the pullback of an $\L$-flat extension is again $\L$-flat.

Such questions have been studied first in a homotopical context by Berrick and Farjoun, \cite{MR1997044}. For homotopical localization functors in the category of topological spaces (in the sense of Bousfield, \cite{MR1481817}, see also Farjoun's book \cite{MR1392221}), preservation of $\L$-flatness (for fiber sequences) under pullbacks was shown to be equivalent for $\L$ to be a so-called nullification functor. The situation is surprisingly more delicate in the category of groups. Farjoun and the second author proved for example that all nilpotent quotient functors have this nice property, which they called \emph{conditional flatness}, see \cite{FS}.

The standard strategy to establish conditional flatness for a localization functor consists in a few reduction steps culminating in a simpler form, which Gran identified as admissibility in the sense of Galois for the class of regular epimorphisms \cite[Proposition~3.3]{MR1312759}. This shifted the study of conditional flatness in homotopy theory to that of admissibility in semi-abelian categories, see \cite{GranScherer}. Admissibility had been introduced by Janelidze and Kelly in \cite{MR1312759} and has since then played a central role in the categorical study of extensions, let us mention for example Everaert, Gran, and Van der Linden's work in \cite{EGV}.

In this article we study admissibility for localization functors in the category of crossed modules (of groups), a category of interest to both topologists due to Whitehead's work on connected $2$-types, \cite{whitehead}, and algebraists since Brown and Spencer \cite{BS} proved the equivalence between crossed modules
and internal groupoids in the category of groups (a result that they credit
to Verdier). This equivalence relates two interesting notions and allows one to deal with the concept of internal groupoid in an alternative way, that is useful for computations.
Moreover, crossed modules form a semi-abelian category in the sense of Janelidze, M\'arki and Tholen, \cite{JMT}.
We adopt the algebraic point of view here and continue our work started in \cite{MSS}. Indeed, among the reduction steps we have mentioned above, the first one calls on fiberwise localization techniques. For group theoretical localization and homotopy localization functors, it allows one to reduce the study to extensions with local kernel (fiber). Fiberwise localization techniques are available in the category of groups thanks to work of Casacuberta and Descheemaeker, \cite{Casacuberta}, but we proved in \cite{MSS} that they are not at hand in general for crossed modules. Our aim in this article is thus to modify the strategy to be able to study admissibility in this setting.

We focus on localization functors such that the co-augmentation morphism $\ell^\T\colon \T \to \L\T$ is a regular epimorphism for all crossed modules $\T$. We call them \emph{regular-epi localization} and notice that many examples of interest are provided by nullification functors, as defined in \cref{null}. Any crossed module $\A$ determines a nullification functor $\P_\A$ that ``kills'' all morphisms from $\A$ and there are other regular-epi localization functors such as abelianization. One first important observation which makes the reduction strategy viable is that, even though fiberwise localization does not exist in general, even for nullification functors, we can use this tool for \emph{certain} extensions.

\medskip

\noindent
{\bf Lemma~\ref{lem:fiberwiseforpb}.}
 \emph{Let $\L$ be a regular-epi localization.  Let\begin{equation}
\begin{tikzpicture}[descr/.style={fill=white},scale= 0.8, baseline=(A.base)] 
\node (A) at (0,0) {$\T$};
\node (B) at (2.5,0) {$\Q$};
\node (C) at (-2.5,0) {$\N$};
\node (O1) at (-5,0) {$1$};
\node (O2) at (5,0) {$1$};
\path[-stealth,font=\scriptsize]
(C.east) edge node[above] {$\kappa$} (A.west)
(O1) edge node[above] {$ $} (C)
(B) edge node[above] {$ $} (O2)
(A) edge node[above] {$\alpha$} (B);
\end{tikzpicture} 
 \end{equation} be an $\L$-flat exact sequence of crossed modules and $g \colon \Q' \to \Q$ a morphism of crossed modules. Then, we can construct the fiberwise localization of the pullback of (1) along $g$:
 \[ 
 \begin{tikzpicture}[descr/.style={fill=white},scale = 0.8,baseline=(A.base)]
\node (A) at (0,0) {$\N$};
\node (B) at (0,2) {$\N$};
\node (C) at (2.5,2) {$\T'$};
\node (D) at (2.5,0) {$\T$};
\node (E) at (5,2) {$\Q'$};
\node (F) at (5,0) {$\Q$};
\node (O1) at (-2.5,2) {$1$};
\node (O1') at (-2.5,0) {$1$};
\node (O2) at (7.5,2) {$1$};
\node (O2') at (7.5,0) {$1$};
  \path[-stealth,font=\scriptsize]
(O1) edge node[above] {$ $} (B)
(E) edge node[above] {$ $} (O2)
(F) edge node[above] {$ $} (O2')
(O1') edge node[above] {$ $} (A)
  (A.east) edge node[below] {$\kappa$} (D.west) 
 (C.south) edge node[right] {$\pi_{\T}$} (D.north) 
 (B.east)  edge node[above] {$\kappa'$} (C.west)
 (E.south) edge node[right] {$g$} (F.north)
 (C.east)  edge node[above] {$\pi_{\Q'}$} (E.west)
 (D.east) edge node[below] {$\alpha$} (F.west);
  \draw[commutative diagrams/.cd, ,font=\scriptsize]
(B) edge[commutative diagrams/equal] (A) ;
\end{tikzpicture}
\]
}

This allows us to relate conditional flatness with admissibility, in the same spirit as what was done in the category of groups, \cite{FS}, or in the wider context of semi-abelian categories where fiberwise localization exists, \cite{GranScherer}.
A localization functor $\L$ is said to be \emph{admissible} for the class of regular epimorphisms if it preserves any pullback of the form \[
\begin{tikzpicture}[descr/.style={fill=white},yscale=0.8]
\node (A) at (0,0) {$\LT$};
\node (B) at (0,2) {$\T'$};
\node (C) at (2,2) {$\Q$};
\node (D) at (2,0) {$\LQ$};
  \path[-stealth,font=\scriptsize]
 (B.south) edge node[left] {$\pi_\LT$} (A.north) 
 (C.south) edge node[right] {$\ell^\Q$} (D.north) 
  (B.east) edge node[above] {$\pi_{\Q}$} (C.west) 
 (A.east) edge node[below] {$\alpha$} (D.west);
\end{tikzpicture} 
\] where $\alpha$ is a regular epimorphism between $\L$-local objects. 

\medskip

\noindent
{\bf Theorem~\ref{finalstep}.}
\emph{
Let $\L$ be a regular-epi localization functor. Then the following statements are equivalent
\begin{enumerate}
\item $\L$ is conditionally flat;
\item $\L$ is admissible for the class of regular epimorphisms.
\end{enumerate}
}
\medskip

One difference between groups and crossed modules and maybe the main source of complication is highlighted by the behavior of kernels. This was already the reason why one cannot always construct fiberwise localization and we were also surprised to find examples of nullification functors for which the kernel of the nullification morphism $\ell^\T\colon \T \to \PAT$ is not always $\P_\A$-acyclic, see \cite[Proposition~4.6]{MSS}. For groups and spaces, this property actually characterizes nullification functors.

Still we prove here that acyclic kernels implies admissibility and in \cref{prop:acyclicimpliesnullification}, that if the kernels of the localization morphisms are $\L_f$-acyclic, then $\L_f$ is a nullification functor. Well behaved nullification functors are therefore admissible, but what about arbitrary nullification functors, for which fiberwise localization does not necessarily exist and for which the kernel of the nullification is not necessarily acyclic? By carefully looking at the inductive construction of $\PAT$ we show our main result, namely that all nullification functors are admissible.

\medskip

\noindent
{\bf Theorem~\ref{thm:PAadmissible}.}
\emph{Let $\A$ be any crossed module. The nullification functor $\P_\A$ is admissible for the class of regular epimorphisms.} 

\medskip

We end this introduction with a short outline. 
The first section consists of preliminaries that we use in the rest of the article. Then in Section~2 we introduce $\L$-flat exact sequences and conditionally flat localization functors in the context of crossed modules. We show how to construct fiberwise localization of $\L$-flat exact sequences. The third section is essential in the development of a simpler characterisation of conditional flatness: It provides an equivalence with the notion of admissibility in the specific context of regular-epi localization functors. In Section~4 the link between $\L$-acyclicity and admissibility is established and the last section is devoted to the proof that every nullification functor is admissible.

\medskip

\noindent
{\bf Acknowledgments.} We would like to thank Marino Gran for sharing his insight about admissibility.


\section{Preliminaries}

\subsection{The semi-abelian category of crossed modules}
In this subsection, following Norrie \cite{Norrie} and Brown-Higgins \cite{BH}, we provide the basic definitions and notation concerning crossed modules.

\begin{definition}\label{xmod grp}\cite{whitehead}
A \emph{crossed module of groups} is a pair of groups $T_1$ and $T_2$, an action by group automorphisms of $T_2$ on $T_1$, denoted by $T_2 \times T_1 \to T_1 \colon (b,t) \mapsto \;^{b}t$, together with a group homomorphism $\partial^\T \colon T_1 \to T_2$ such that for any $b$ in $T_2$ and any $t$, $s$ in $T_1$,
\begin{equation}
\partial^\T(\;^b t) = b\partial^\T(t)b^{-1},
\end{equation}
\begin{equation}\label{peiffer}
^{\partial^\T(t)}s = tst^{-1}.
\end{equation}
\end{definition}

Hence we often write a crossed module as a triple $(T_1,T_2,\partial^\T)$, or simply $\T$ for short, and we refer sometimes to $\partial^\T$ as the \emph{connecting morphism}. 

\begin{definition} \label{morphxmod}
Let $\N := (N_1,N_2,\partial^{\N})$ and $\M := (M_1,M_2,\partial^{\M})$ be two crossed modules. A \emph{morphism of crossed modules} $\alpha\colon \N \to \M$ is a pair of group homomorphisms $\alpha_1\colon N_1 \rightarrow M_1$ and $\alpha_2\colon N_2 \rightarrow M_2$ such that the two following diagrams commute
\[
\begin{tikzpicture}[descr/.style={fill=white},scale=0.8,baseline=(A.base)]
\node (A) at (0,0) {$N_2$};
\node (B) at (0,2) {$N_1$};
\node (C) at (2,2) {$M_1$};
\node (D) at (2,0) {$M_2$};
  \path[-stealth,font=\scriptsize]
 (B.south) edge node[left] {$\partial^{\N}$} (A.north) 
 (C.south) edge node[right] {$\partial^{\M}$} (D.north) 
  (B.east) edge node[above] {$\alpha_1$} (C.west) 
 (A.east) edge node[below] {$\alpha_2$} (D.west);
\end{tikzpicture}\hspace{1cm} \begin{tikzpicture}[descr/.style={fill=white},yscale=0.8,baseline=(A.base)]
\node (A) at (0,0) {$M_2 \times M_1$};
\node (B) at (0,2) {$N_2 \times N_1$};
\node (C) at (2,2) {$N_1$};
\node (D) at (2,0) {$M_1.$};
  \path[-stealth,font=\scriptsize]
 (B.south) edge node[left] {$(\alpha_2, \alpha_1)$} (A.north) 
 (C.south) edge node[right] {$\alpha_1$} (D.north) 
  (B.east) edge node[above] {$ $} (C.west) 
 (A.east) edge node[below] {$ $} (D.west);
\end{tikzpicture}
\]
where the horizontal arrows in the diagram on the right are the respective group actions of the two crossed modules.
\end{definition}

We write $\mathsf{XMod}$ for the category of crossed modules of groups. 

\begin{remark}\label{lemma_lefttruncation}
    There is an embedding of the category of groups in this category via two functors which are respectively left and right adjoint to the truncation functor $Tr \colon \sf XMod \to Grp$ that sends a crossed module $\T := (T_1,T_2,\partial^\T)$ to $T_2$. The functor $\sf X\colon Grp \to XMod$ which sends a group $G$ to the crossed module ${\sf X} G = ( 1,G, 1)$ reduced to the group $G$ at level~$2$ is the left adjoint functor and the functor ${ \sf R \colon  Grp \to XMod } \colon G \mapsto (G,G,Id_G) $ is the right ajoint functor. This will help us to import group theoretical results into $\sf XMod$.
\end{remark}



There is an obvious notion of subcrossed module, see \cite{Norrie}. One simply requires the subobject to be made levelwise of subgroups, the connecting
homomorphism and the action are induced by the given connecting homomorphism and action. The notion of normality is less obvious.


\begin{definition}\label{defnormalcrossed}
A subcrossed module $\N := (N_1,N_2,\partial^\N)$ of $\T := (T_1,T_2,\partial^\T)$ is \emph{normal} if the following three conditions hold
\begin{enumerate}
    \item $N_2$ is a normal subgroup of $T_2$;
    \item for any $t_2 \in T_2$ and $n_1 \in N_1$, we have $^{t_2}n_1 \in N_1$;
    \item $[N_2,T_1]:= \langle \;^{n_2}t_1t_1^{-1} \mid  t_1  \in T_1, n_2 \in N_2 \rangle \subseteq N_1$.
\end{enumerate}
\end{definition}


In contrast to limits, which are built component-wise, colimits are generally more delicate to construct. In particular, the construction of cokernels is not straightforward,
but when ${\sf N}$ is a normal subcrossed module of $\sf T$ the cokernel is simply the levelwise quotient by the normal subgroups $N_1 \triangleleft T_1$ and $N_2 \triangleleft T_2$. \newline

The category of crossed modules shares many nice properties with the category of groups. The traditional homological lemmas, \cite{BB}, the Split Short Five Lemma, \cite{Bourn}, and the Noether Isomorphism Theorems, \cite{BB}, hold. One can recognize pullbacks by looking at kernels or cokernels, \cite[Lemmas 4.2.4 and 4.2.5]{BB}, and in fact $\sf Xmod$ is a semi-abelian category, as introduced by Janelidze, M\'arki, and Tholen in \cite{JMT}. This is shown in \cite{JMT}. There is one result we will use several times in this article, namely \cite[Lemma~4.2.4]{BB}, which we recall now.

\begin{prop}\label{usefulproposemiab}
Let $\mathcal{C}$ be a semi-abelian (or homological) category. Consider the following diagram of exact rows:
\[
\begin{tikzpicture}[descr/.style={fill=white},baseline=(A.base),scale=0.8] 
\node (A) at (0,0) {$T'$};
\node (B) at (2.5,0) {$Q'$};
\node (C) at (-2.5,0) {$N'$};
\node (A') at (0,-2) {$T$};
\node (B') at (2.5,-2) {$Q$};
\node (C') at (-2.5,-2) {$N$};
\node (O1) at (-4.5,0) {$1$};
\node (O1') at (-4.5,-2) {$1$};
\node (O2) at (4.5,0) {$1$};
\node (X) at (1.25,-1) {$(2)$};
\path[-stealth,font=\scriptsize]
(O1) edge node[above] {$ $} (C)
(B) edge node[above] {$ $} (O2)
(O1') edge node[above] {$ $} (C')
(B.south) edge node[right] {$ w$}  (B'.north)
 (C.south) edge node[left] {$ u $}  (C'.north)
(A.south) edge node[left] {$ v$} (A'.north)
(C'.east) edge node[above] {$\kappa$} (A'.west)
(A') edge node[above] {$\alpha$} (B') ([yshift=2pt]A'.east)
(C.east) edge node[above] {$\kappa'$} (A.west)
(A) edge node[above] {$\alpha'$} (B);
\end{tikzpicture} 
\]

Then the following statements hold.
\begin{enumerate}
    \item[(1)] If $u$ is an isomorphism then (2) is a pullback.
    \item[(2)] If $u$ and $w$ are regular epimorphisms then $v$ is also a regular epimorphism.
\end{enumerate}
\end{prop}
\subsection{Localization functors}
In this subsection we recall the definition of localization functors in the category of crossed modules. We also recall some important properties of such functor as well as some examples.

\begin{definition}\label{localization}
A \emph{localization} functor in the category of crossed modules is a coaugmented idempotent functor $\sf L \colon XMod \to XMod$. The coaugmentation $\ell \colon {\sf Id} \rightarrow \L$ is a natural transformation such that $\ell^{\LX}$ and $ \L \ell^\X$ are isomorphisms.
\end{definition}

In particular we have $\ell^{\LX} = \L \ell^\X$, see \cite[Proposition~1.1]{MR1796125}.

\begin{definition}\label{local}
Let $\sf L$ be a localization functor. A crossed module $\sf T$ is $\sf L$-\emph{local}
if $\ell^{\sf T}\colon \sf T \rightarrow \sf L \sf T$ is an
isomorphism. A morphism $f\colon \sf N \rightarrow \sf M$ is an $\sf L$-\emph{equivalence} 
if ${\sf L} f$ is an isomorphism.
\end{definition}

We recall a few basic and useful closure properties of $\L$-equivalences.

\begin{lemma}\label{lemma_colimitandequivalences}
\begin{enumerate}
    \item The pushout of an $\sf L$-equivalence is an $\L$-equivalence.
     \item The composition of $\L$-equivalences is an $\L$-equivalence.
    \item A $\kappa$-filtered colimit of a diagram $\sf T_\beta$ of $\sf L$-equivalences $\sf T_\beta \rightarrow  T_{\beta +1}$ for all successor ordinals 
    $\beta +1 < \kappa$ yields an $\sf L$-equivalence $\sf T_0 \rightarrow T_\kappa = \sf colim_{\beta < \kappa} T_\beta$.
    \item Let ${\sf F}$ be an $I$-indexed diagram of $\sf L$-equivalences in the category of morphisms of crossed modules. 
    Then the colimit $\sf colim_I F$ is an $\L$-equivalence.
\end{enumerate}
\end{lemma}


Sometimes a localization functor $\L$ is associated to a full reflexive subcategory $\mathcal L$ of $\sf XMod$.
The pair of adjoint functors ${\sf U\colon \mathcal L \leftrightarrows \sf XMod\colon \sf F}$ provides a localization functor $\sf L = \sf FU$, as Cassidy, H\'ebert, and Kelly do in \cite{MR779198}.
Some other times there is a morphism $f$ one wishes to invert so as to construct a localization functor often written ${\sf L}_f$.

\begin{definition}\label{f-local}
Let $f$ be a morphism of crossed modules. A crossed module
$\sf T$ is \emph{${\sf L}_f$-local} if $Hom(f, \T)$ is an isomorphism.
A morphism $g$ in $\sf XMod$ is an \emph{${\sf L}_f$-equivalence} if
$Hom(g, \T)$ is an isomorphism for any ${\sf L}_f$-local crossed module~$\sf T$.
\end{definition}

Such localization functors exist in $\sf XMod$, see for example
Bousfield's foundational work \cite{MR478159}.
Local objects and local equivalences coincide then with
the notions introduced in Definition~\ref{local}. Proposition~\ref{lemma_colimitandequivalences} is the analogue of Hirschhorn's \cite[Proposition~1.2.20 and Proposition 1.2.21]{Hirschhorn}. 


If the codomain of the morphism $f$ is the trivial crossed module, the functor $\L_f$ is of particular interest.

\begin{definition}\label{null}
Let $\sf A$ be a crossed module and $f$ be the morphism $\sf A \to 1$. The localization functor ${\sf L}_f$ is then written $\sf P_A$ and is called a \emph{nullification} functor. An $f$-local object is called $\sf A$-\emph{null}, or $\sf A$-\emph{local} and a crossed module $\sf T$ is $\sf A$-\emph{acyclic} if $\sf P_A T = 1$. The localization morphism $\ell^T\colon \T \to \PAT$ is written $p^\T$.
\end{definition}


\begin{prop}\label{propnullify}
Let $\sf A$ and $\sf T$ be crossed modules. Then there exists an ordinal $\lambda$ depending on $\sf A$ such that $\sf P_A T$ is
constructed as a transfinite filtered colimit of a diagram of the form 
$\sf T= \sf T_0 \to T_1 \rightarrow 
\dots \rightarrow T_\beta \rightarrow \dots$ for $\beta < \lambda$ where all morphisms are $\sf P_A$-equivalences and regular epimorphims.
\end{prop}

This inductive construction has been carefully described in \cite[Proposition~2.8]{MSS}. The reason why each step is a $\sf P_A$-equivalence and a regular epimorphism is that $\T_{\beta +1}$ is constructed from $\T_\beta$ by taking the cokernel of all morphisms $\A \to \T_\beta$. We recall the details and use them in \cref{sec:nullification and admissibility}.
There is a larger class of localization functors we investigate in this sequel to \cite{MSS}. They share with $\sf P_A$ the property that the localization morphism is a regular epimorphism.
\begin{definition}\label{defregepi}
A localization functor $\L$ is a \emph{regular-epi localization} if for any crossed module $\T$ the coaugmentation $\ell^\T \colon \T \to \L\T$ is a regular epimorphism. 
\end{definition} 
\begin{remark}
\label{rem:eprireg}
In the category of crossed modules, a morphism $\alpha =(\alpha_1,\alpha_2)$ is a regular epimorphism (a coequalizer of a pair of parallel arrows) if and only if both $\alpha_1$ and $\alpha_2$ are surjective group homomorphisms \cite[Proposition~2.2]{LLR}. A surjective homomorphism of crossed modules is an epimorphism but there exist epimorphisms that are not surjective. In a pointed protomodular category such as $\sf XMod$, regular epimorphisms and normal epimorphisms (the cokernel of some morphism) coincide.
\end{remark}

We present now some interesting examples of localization functors that will illustrate our results in the rest of the article, see also the end of \cite[Section~2]{MSS}.

\begin{example}\label{example_PZ}
The nullification functor $\P_{\sf X\mathbb{Z}}$ with respect to
the crossed module $\sf X\mathbb{Z}$ is given by: 
\[ 
\P_{\sf X\mathbb{Z}} \left( 
\begin{tikzpicture}[descr/.style={fill=white},yscale=0.7,baseline=(O.base)] 
\node (O) at (-5,-1) {$ $};
\node (F) at (-4.5,0) {$N_1$};
\node (F') at (-4.5,-2) {$N_2$};
\path[-stealth,font=\scriptsize]
(F.south) edge node[left] {$ \partial $} (F'.north);
\end{tikzpicture} 
 \right) = 
\begin{tikzpicture}[descr/.style={fill=white},yscale=0.7,baseline=(O.base)] 
\node (O) at (-1,-1) {$ $};
\node (C) at (-2.5,0) {$ N_1 / [N_2, N_1]$};
\node (C') at (-2.5,-2) {$1$};
\path[-stealth,font=\scriptsize]
(C.south) edge node[left] {$  $} (C'.north);
\end{tikzpicture} 
\]
\end{example}

\begin{example}\label{Ab}
The abelianization functor $\sf Ab \colon \sf  XMod \to XMod$ is already described in \cite{Norriethesis}.
It is defined by:
\[ 
{\sf Ab} \left( 
\begin{tikzpicture}[descr/.style={fill=white},yscale=0.7,baseline=(O.base)] 
\node (O) at (-5,-1) {$ $};
\node (F) at (-4.5,0) {$N_1$};
\node (F') at (-4.5,-2) {$N_2$};
\path[-stealth,font=\scriptsize]
(F.south) edge node[left] {$ \partial $} (F'.north);
\end{tikzpicture} 
 \right) = 
\begin{tikzpicture}[descr/.style={fill=white},yscale=0.7,baseline=(O.base)] 
\node (O) at (-1,-1) {$ $};
\node (C) at (-2.5,0) {$N_1/ [N_2,N_1] $};
\node (C') at (-2.5,-2) {$N_2/ [N_2,N_2]$};
\path[-stealth,font=\scriptsize]
(C.south) edge node[left] {$  \tilde{\partial} $} (C'.north);
\end{tikzpicture} 
  \]
  \end{example}

\begin{example}\label{I}
Our third and last example of localization functor of crossed modules is ${ \sf I} \colon \sf XMod \to XMod$, see \cite[Example~2.15]{MSS}:
 \[ 
{ \sf I} \left( 
\begin{tikzpicture}[descr/.style={fill=white},yscale=0.7,baseline=(O.base)] 
\node (O) at (-5,-1) {$ $};
\node (F) at (-4.5,0) {$N_1$};
\node (F') at (-4.5,-2) {$N_2$};
\path[-stealth,font=\scriptsize]
(F.south) edge node[left] {$ \partial^\N $} (F'.north);
\end{tikzpicture} 
 \right) = 
\begin{tikzpicture}[descr/.style={fill=white},yscale=0.7,baseline=(O.base)] 
\node (O) at (-1,-1) {$ $};
\node (C) at (-2.5,0) {$N_2$};
\node (C') at (-2.5,-2) {$N_2$};
\path[-stealth,font=\scriptsize]
(C.south) edge node[left] {$  Id_{N_2}$} (C'.north);
\end{tikzpicture} 
  \]
This functor is induced by the adjunction between the truncation functor $Tr \colon {\sf XMod \to Grp}$, defined by $Tr(T_1,T_2,\partial^\T) = T_2$, see \cref{lemma_lefttruncation}, and its right adjoint ${\sf R \colon Grp \to XMod }$ that sends a group $T$ to $(T,T,Id_T)$.
 \end{example}

\begin{remark}
The functor considered in \cref{example_PZ} is a regular-epi localization, since all nullification functors are so. However regular-epi localizations are not nullification functors in general as illustrated by the functor $\sf Ab$ in \cref{Ab}. Indeed, if $\sf Ab$ were a nullification $\P_\A$, then $\A = (A_1, A_2, \partial^\A)$ would be a perfect crossed module, i.e. one such that $\sf Ab (\A) = (1, 1, Id)$. 
In particular, the group $A_2$ would be a perfect group. 
But then $\P_\A (\X S_3) = \X S_3$ since there are no non-trivial homomorphisms from a perfect group to the symmetric group $S_3$. But we know that ${\sf Ab}(\X S_3) = \X C_2$, where $C_2$ is the cyclic group of order two, so abelianization is not a nullification.

We finally note that a localization functor $\L_f$ is a  regular-epi localization functor if $f$ itself is a regular epimorphism, an analogous observation appears in \cite{CasacubertaAnderson} for groups.


\end{remark}



To conclude these preliminaries, let us recall the notion of fiberwise localization. We introduced this for crossed modules in \cite[Definition~3.1]{MSS}, but this is not new, for spaces a good reference is \cite[Section~I.F]{MR1392221}.

\begin{definition}
Let ${ \sf L} \colon {\sf XMod} \to {\sf XMod}$ be a localization functor. An exact sequence 
\[ 
\begin{tikzpicture}\label{exactseq}[descr/.style={fill=white},scale=0.8,baseline=(A.base),
xscale=1.2] 
\node (A) at (0,0) {$\T$};
\node (B) at (2.5,0) {$\Q$};
\node (C) at (-2.5,0) {$\N$};
\node (O1) at (-4.5,0) {$1$};
\node (O2) at (4.5,0) {$1$};
\path[-stealth,font=\scriptsize]
(C.east) edge node[above] {$\kappa$} (A.west)
;
\path[-stealth,font=\scriptsize]
(O1) edge node[above] {$ $} (C)
(B) edge node[above] {$ $} (O2)
(A) edge node[above] {$\alpha$} (B);
\end{tikzpicture} 
\]
admits a fiberwise localization if there exists a commutative diagram of horizontal exact sequences
\begin{center}
\begin{tikzpicture}[descr/.style={fill=white},baseline=(A.base),scale=0.8] 
\node (A) at (0,0) {$\T$};
\node (B) at (2.5,0) {$\Q$};
\node (C) at (-2.5,0) {$\N$};
\node (A') at (0,-2) {$\E$};
\node (B') at (2.5,-2) {$\Q$};
\node (C') at (-2.5,-2) {$\LN$};
\node (O1) at (-4.5,0) {$1$};
\node (O1') at (-4.5,-2) {$1$};
\node (O2) at (4.5,0) {$1$};
\node (O2') at (4.5,-2) {$1$};
\path[-stealth,font=\scriptsize]
(C.east) edge node[above] {$\kappa$} (A.west)
;
\draw[commutative diagrams/.cd, ,font=\scriptsize]
(B) edge[commutative diagrams/equal] (B');
\path[-stealth,font=\scriptsize]
(C'.east) edge node[above] {$j $} (A'.west)
(O1) edge node[above] {$ $} (C)
(B) edge node[above] {$ $} (O2)
(B') edge node[above] {$ $} (O2')
(O1') edge node[above] {$ $} (C')
 (C.south) edge node[left] {$ \ell^{\N} $}  (C'.north)
(A') edge node[above] {$ p$} (B') ([yshift=2pt]A'.east)
(A) edge node[above] {$\alpha$} (B);
\path[-stealth,font=\scriptsize]
(A.south) edge node[left] {$ g $} (A'.north);
\end{tikzpicture} 
 \end{center} where $g$ is an $\L$-equivalence.
\end{definition}

The following theorem is a fusion of two results from \cite{MSS} namely Theorem 3.4 and Corollary 3.7. From now on, every localization functor that we consider is a regular-epi localization.

\begin{theorem}\label{thmfiberwiselocalization}
Let $\sf L \colon \sf XMod \to XMod$ be a regular-epi localization functor. An exact sequence of crossed modules
\begin{equation}
\begin{tikzpicture}[descr/.style={fill=white},baseline=(A.base),scale=0.8] 
\node (A) at (0,0) {$\T$};
\node (B) at (2.5,0) {$\Q$};
\node (C) at (-2.5,0) {$\N$};
\node (O1) at (-4.5,0) {$1$};
\node (O2) at (4.5,0) {$1$};
\path[-stealth,font=\scriptsize]
(C.east) edge node[above] {$\kappa$} (A.west)
;
\path[-stealth,font=\scriptsize]
(O1) edge node[above] {$ $} (C)
(B) edge node[above] {$ $} (O2)
(A) edge node[above] {$\alpha$} (B);
\end{tikzpicture} 
 \end{equation}
admits a fiberwise localization if and only if we have the following inclusion \begin{equation}\label{Conditionnormalite}
[\kappa_2(ker(\ell_2^\N)),T_1] \subseteq \kappa_1(ker(\ell_1^\N))
\end{equation}
\end{theorem}


\section{Fiberwise localization and flatness}

 In this section, we investigate the fiberwise localization of $\L$-flat exact sequences and their pullbacks in the context of regular-epi localization functors of crossed modules $\L \colon \sf XMod \to XMod$ (even if this notion is not defined only for regular-epi functor as we will see in \cref{condflatnotregularepi}). This section will be essential to study the link between conditionally flatness and admissibility in \cref{section CF and Ad}. 
First, let us recall the definitions of $\L$-flat and conditionally flatness.  

\begin{definition}
\label{def:flat}
Let $\L$ be a localization functor, a short exact sequence\[ \begin{tikzpicture}[descr/.style={fill=white},scale=0.8, baseline=(A.base),
xscale=1.2] 
\node (A) at (0,0) {$\T$};
\node (B) at (2.5,0) {$\Q$};
\node (C) at (-2.5,0) {$\N$};
\node (O1) at (-4.5,0) {$1$};
\node (O2) at (4.5,0) {$1$};
\path[-stealth,font=\scriptsize]
(C.east) edge node[above] {$\kappa$} (A.west)
;
\path[-stealth,font=\scriptsize]
(O1) edge node[above] {$ $} (C)
(B) edge node[above] {$ $} (O2)
(A) edge node[above] {$\alpha$} (B);
\end{tikzpicture}\] is called \emph{$\L$-flat} if the sequence \begin{tikzpicture}[descr/.style={fill=white},xscale=0.8,baseline=(A.base),
xscale=1.2] 
\node (A) at (0,0) {$\LT$};
\node (B) at (2.5,0) {$\LQ$};
\node (C) at (-2.5,0) {$\LN$};
\path[-stealth,font=\scriptsize]
(C.east) edge node[above] {$\L(\kappa)$} (A.west)
;
\path[-stealth,font=\scriptsize]
(A) edge node[above] {$\L(\alpha)$} (B);
\end{tikzpicture} is a short exact sequence.
\end{definition}

\begin{remark}
    We recall that limits are computed componentwise in the category of crossed modules. In the case of pullbacks in $\sf XMod$ they are built as follows \cite{LG}. Let $\alpha \colon \T \to \Q$ and $g \colon \Q' \to \Q$ be two morphisms of crossed modules. Then the pullback of $\alpha$ along $g$ is given by the following square
\[
\begin{tikzpicture}[descr/.style={fill=white},yscale=0.8]
\node (A) at (0,0) {$\T$};
\node (B) at (0,2) {$\T'$};
\node (C) at (2,2) {$\Q'$};
\node (D) at (2,0) {$\Q$};
  \path[-stealth,font=\scriptsize]
 (B.south) edge node[left] {$\pi_\T$} (A.north) 
 (C.south) edge node[right] {$g$} (D.north) 
  (B.east) edge node[above] {$\pi_{\Q'}$} (C.west) 
 (A.east) edge node[below] {$\alpha$} (D.west);
\end{tikzpicture}
\]
The object part $\T'$ of the pullback is built component-wise as in the case of groups 
\[ 
(T_1\times_{Q_1}Q'_1, T_2\times_{Q_2}Q'_2, \partial'),
\]
where $\partial'$ and the action are induced by the universal property of the pullbacks in $\sf Grp$. The projections are the natural ones, given also component-wise.
\end{remark}

 Following the terminology introduced in \cite{FS} for groups and spaces, we define the notion of conditional flatness for localization functors in crossed modules.


\begin{definition}\label{defcondflat}
Let $\L$ be a localization functor. We say that this functor is \emph{conditionally flat} if the pullback of any $\L$-flat exact sequence is $\L$-flat.
\end{definition}

 In \cref{section CF and Ad} we provide a characterization of conditional flatness. To achieve this goal we will use a similar strategy to the one applied to groups and topological spaces in \cite{FS}. The authors exploit heavily the existence of fiberwise localization in the categories of groups and spaces. However, in our article \cite{MSS}, we observed that fiberwise localization does not always exist for a given localization functor and a given exact sequence in $\sf XMod$.
Fortunately, when we work with $\L$-flat exact sequences we can show that it is always possible to construct a fiberwise localization.


 %

\begin{lemma}
\label{Lflatfiberwise}
 Let $\L$ be a regular-epi localization. Then any $\L$-flat exact sequence of crossed modules admits a fiberwise localization.
\end{lemma}

\begin{proof}
Let \begin{tikzpicture}[descr/.style={fill=white},scale= 0.8, baseline=(A.base)] 
\node (A) at (0,0) {$\T$};
\node (B) at (2.5,0) {$\Q$};
\node (C) at (-2.5,0) {$\N$};
\node (O1) at (-5,0) {$1$};
\node (O2) at (5,0) {$1$};
\path[-stealth,font=\scriptsize]
(C.east) edge node[above] {$\kappa$} (A.west)
(O1) edge node[above] {$ $} (C)
(B) edge node[above] {$ $} (O2)
(A) edge node[above] {$\alpha$} (B);
\end{tikzpicture}
be an $\L$-flat exact sequence of crossed modules. The $\L$-flatness of the sequence implies in particular that ${\sf L} \kappa$ is a monomorphism. Consider the following diagram of exact sequences:
\[ \begin{tikzpicture}[descr/.style={fill=white},baseline=(A.base),yscale=0.8]
\node (O1) at (-2,2) {$ 1$};
\node (O2) at (-2,0) {$ 1$};
\node (A) at (0,0) {$\ker(\ell^\T)$};
\node (X) at (1,1) {$(1)$};
\node (B) at (0,2) {$\ker(\ell^\N)$};
\node (C) at (2,2) {$\N$};
\node (D) at (2,0) {$\T$};
\node (E) at (4,2) {$\LN$};
\node (F) at (4,0) {$\LT$};
  \path[>-stealth,font=\scriptsize]
  (A.east) edge node[below] {$ $} (D.west) 
 (C.south) edge node[right] {$ \kappa $} (D.north) 
 (B.south) edge node[left] {$ $} (A.north) 
 (B.east)  edge node[above] {$ $} (C.west)
 (E.south) edge node[right] {$ \L \kappa$} (F.north);
   \path[-stealth,font=\scriptsize]
   (O1) edge node[below] {$ $} (B)
   (O2) edge node[below] {$ $} (A)
    (C.east)  edge node[above] {$ \ell^\N$} (E.west)
 (D.east) edge node[below] {$\ell^\T$} (F.west);
\end{tikzpicture}
\]
We conclude from \cite[Lemma~4.2.4.(1)]{BB} that (1) is a pullback since ${\sf L} \kappa$ is a monomorphism. Then we have that $\kappa(\ker(\ell^\N))$ is a normal subcrossed 
module of $\T$ as it can be seen as the intersection of the normal subcrossed modules $\N$ and $\ker(\ell^\T)$ of $\T$. Therefore, we can apply \cref{thmfiberwiselocalization} 
 \end{proof}

 To understand conditional flatness we must study 
the pullback of an $\L$-flat exact sequence. It will thus be very handy in \cref{section CF and Ad} to know that any such pullback admits a fiberwise localization.

\begin{lemma}
\label{lem:fiberwiseforpb}
 Let $\L$ be a regular-epi localization.  Let\begin{equation}\label{eseq L-flat}
\begin{tikzpicture}[descr/.style={fill=white},scale= 0.8, baseline=(A.base)] 
\node (A) at (0,0) {$\T$};
\node (B) at (2.5,0) {$\Q$};
\node (C) at (-2.5,0) {$\N$};
\node (O1) at (-5,0) {$1$};
\node (O2) at (5,0) {$1$};
\path[-stealth,font=\scriptsize]
(C.east) edge node[above] {$\kappa$} (A.west)
(O1) edge node[above] {$ $} (C)
(B) edge node[above] {$ $} (O2)
(A) edge node[above] {$\alpha$} (B);
\end{tikzpicture} 
 \end{equation} be an $\L$-flat exact sequence of crossed modules and $g \colon \Q' \to \Q$ a morphism of crossed modules. Then, we can construct the fiberwise localization of the pullback of \eqref{eseq L-flat} along $g$
 \[ \begin{tikzpicture}[descr/.style={fill=white},scale = 0.8,baseline=(A.base)]
\node (A) at (0,0) {$\N$};
\node (B) at (0,2) {$\N$};
\node (C) at (2.5,2) {$\T'$};
\node (D) at (2.5,0) {$\T$};
\node (E) at (5,2) {$\Q'$};
\node (F) at (5,0) {$\Q$};
\node (O1) at (-2.5,2) {$1$};
\node (O1') at (-2.5,0) {$1$};
\node (O2) at (7.5,2) {$1$};
\node (O2') at (7.5,0) {$1$};
  \path[-stealth,font=\scriptsize]
(O1) edge node[above] {$ $} (B)
(E) edge node[above] {$ $} (O2)
(F) edge node[above] {$ $} (O2')
(O1') edge node[above] {$ $} (A)
  (A.east) edge node[below] {$\kappa$} (D.west) 
 (C.south) edge node[right] {$\pi_{\T}$} (D.north) 
 (B.east)  edge node[above] {$\kappa'$} (C.west)
 (E.south) edge node[right] {$g$} (F.north)
 (C.east)  edge node[above] {$\pi_{\Q'}$} (E.west)
 (D.east) edge node[below] {$\alpha$} (F.west);
  \draw[commutative diagrams/.cd, ,font=\scriptsize]
(B) edge[commutative diagrams/equal] (A) ;
\end{tikzpicture}\]
 \end{lemma}

 \begin{remark} \label{kappakappa'}
 In the rest of the article, and in particular in the following proof, we identify $\N$ with the normal subcrossed module $\kappa(\N)$ of $\T$ and with $\kappa'(\N)$, normal subcrossed module of $\T'$. We will therefore omit the us of $\kappa$ and $\kappa'$. For example an element of the group $N_1$ that we want to consider in $\T'_1$ will be denoted $(n_1,1)$ instead of $\kappa'_1(n_1)=(\kappa_1(n_1),1)$.
 \end{remark}

 \begin{proof}[Proof of \cref{lem:fiberwiseforpb}]
 We need to verify that $\ker(\ell^\N)$ is a normal crossed module of $\T'$. 
Since $\N$ is a subcrossed module of $\T'$, we just need to verify \eqref{Conditionnormalite} of \cref{thmfiberwiselocalization}. Let $(t_1,q_1)$ be an element in $T'_1$ and $(x_2,1)$ be an element of $\ker(\ell_2^\N)$, then we have the following equality
\[ \;^{(x_2,1)}(t_1,q_1)(t_1,q_1)^{-1} = (\;^{x_2}t_1t_1^{-1},q_1q_1^{-1})= (\;^{x_2}t_1t_1^{-1},1). \]
Indeed, by \cref{Lflatfiberwise} we know that the original sequence \eqref{eseq L-flat} admits a fiberwise localization which then implies by \cref{thmfiberwiselocalization} that $[ker(\ell_2^\N),T_1] \subset ker(\ell_1^\N)$ i.e for any $x_2 \in ker(\ell_2^\N)$ and $t_1 \in T_1$ we have  $^{x_2}t_1t_1^{-1} \in ker(\ell_1^\N)$. But then, with the notation introduced in \cref{kappakappa'}, this is equivalent to say that the element $(\;^{x_2}t_1t_1^{-1},1)$ belongs to $ker(\ell_1^\N)$.
\end{proof}

This lemma is not trivial since the fiberwise localization of an exact sequence of crossed modules does not always exist as we have proved in \cite[Theorem~4.5]{MSS}. 
If we want the strategy for groups and spaces to be also viable in the study of conditional flatness for crossed modules, we need a final ingredient, namely a commutation rule for the fiberwise localization and the pullback operations.


\begin{prop}
\label{prop:pullbackfiber=fiberpullback}
 Let us consider an $\L$-flat exact sequence where $\L$ is a regular-epi localization functor. Then, the pullback of its fiberwise localization is the fiberwise localization of its pullback. 
\end{prop}

\begin{proof}
Let  \[ \begin{tikzpicture}[descr/.style={fill=white},baseline=(A.base), yscale=0.8]
\node (A) at (0,0) {$\N$};
\node (B) at (0,2) {$\N$};
\node (C) at (2,2) {$\T'$};
\node (D) at (2,0) {$\T$};
\node (E) at (4,2) {$\Q'$};
\node (F) at (4,0) {$\Q$};
  \node (O1) at (-2,2) {$1$};
\node (O1') at (-2,0) {$1$};
\node (O2) at (6,2) {$1$};
\node (O2') at (6,0) {$1$};
  \path[-stealth,font=\scriptsize]
  (O1) edge node[above] {$ $} (B)
(E) edge node[above] {$ $} (O2)
(F) edge node[above] {$ $} (O2')
(O1') edge node[above] {$ $} (A)
  (A.east) edge node[below] {$\kappa$} (D.west) 
 (C.south) edge node[right] {$\pi_{\T}$} (D.north) 
 (B.east)  edge node[above] {$\kappa'$} (C.west)
 (E.south) edge node[right] {$g$} (F.north)
 (C.east)  edge node[above] {$\pi_{\Q'}$} (E.west)
 (D.east) edge node[below] {$\alpha$} (F.west);
  \draw[commutative diagrams/.cd, ,font=\scriptsize]
(B) edge[commutative diagrams/equal] (A) ;
\end{tikzpicture}\] be the pullback of an $\L$-flat exact sequence. Then we construct the fiberwise localizations of the two sequences by quotienting out the kernel of the localization morphism $\ell^\N$ as follows.
\[
\begin{tikzpicture}[yscale=0.8]
	\begin{pgfonlayer}{nodelayer}
		\node [style=green] (0) at (-6.75, 1.25) {$\N$};
		\node [style=green] (0') at (-8.5, 1.25) {$1$};
		\node [style=green] (1) at (-3.75, 1.25) {$\T'$};
		\node [style=green] (2) at (-0.75, 1.25) {$\Q'$};
		\node [style=green] (2') at (1, 1.25) {$1$};
		\node [style=green] (3) at (-6.75, -1.75) {$\N$};
		\node [style=green] (3') at (-8.5, -1.75) {$1$};
		\node [style=green] (4) at (-3.75, -1.75) {$\T$};
		\node [style=green] (5) at (-0.75, -1.75) {$\Q$};
		\node [style=green] (5') at (1, -1.75) {$1$};
		\node [style=green] (6) at (-4.75, 0.25) {$\LN$};
		\node [style=green] (6') at (-6.5, 0.25) {$1$};
		\node [style=green] (7) at (-1.75, 0.25) {$\T' / \ker(\ell^\N)$};
		\node [style=green] (8) at (1.25, 0.25) {$\Q'$};
		\node [style=green] (8') at (3, 0.25) {$1$};
		\node [style=green] (9) at (-4.75, -2.75) {$\LN$};
		\node [style=green] (9') at (-6.5, -2.75) {$1$};
		\node [style=green] (10) at (-1.75, -2.75) {$\T / \ker(\ell^\N)$};
		\node [style=green] (11) at (1.25, -2.75) {$\Q$};
		\node [style=green] (11') at (3, -2.75) {$1$};
		\node [style=none,font=\scriptsize] (12) at (-5.25, 1.5) {$\kappa'$};
		\node [style=none,font=\scriptsize] (13) at (-2, 1.5) {$\pi_{\Q'}$};
		\node [style=none,font=\scriptsize] (16) at (-5.25, -1.5) {$\kappa$};
		\node [style=none,font=\scriptsize] (17) at (-2, -1.5) {$\alpha$};
		\node [style=none,font=\scriptsize] (18) at (-0.5, -0.5) {$g$};
		\node [style=none,font=\scriptsize] (19) at (-3.25, -0.5) {$\pi_{\T}$};
		\node [style=none,font=\scriptsize] (21) at (-3.25, -3) {$j$};
		\node [style=none,font=\scriptsize] (22) at (-3.5, 0.5) {$j'$};
		\node [style=none,font=\scriptsize] (23) at (0, -3) {$p$};
		\node [style=none,font=\scriptsize] (24) at (-0.25, 0.5) {$p'$};
		\node [style=none,font=\scriptsize] (26) at (1.5, -1.25) {$g$};
		\node [style=none,font=\scriptsize] (29) at (-2.75, 1) {$f'$};
		\node [style=none,font=\scriptsize] (29) at (-2.7, -2) {$f$};
		\node [style=none,font=\scriptsize] (30) at (-5.5, 1) {$\ell^{\N}$};
		\node [style=none,font=\scriptsize,font=\scriptsize] (31) at (-5.5, -2) {$\ell^{\N}$};
	\end{pgfonlayer}
		\draw[commutative diagrams/.cd, ,font=\scriptsize]
(0) edge[commutative diagrams/equal] (3);
	\draw[commutative diagrams/.cd, ,font=\scriptsize]
(5) edge[commutative diagrams/equal] (11);
	\draw[commutative diagrams/.cd, ,font=\scriptsize]
(6) edge[commutative diagrams/equal] (9);
	\draw[commutative diagrams/.cd, ,font=\scriptsize]
(2) edge[commutative diagrams/equal] (8);
	\begin{pgfonlayer}{edgelayer}
	    \draw [style=stealth] (0') to (0);
	    \draw [style=stealth] (3') to (3);
	     \draw [style=stealth] (2) to (2');
	     \draw [style=stealth] (5) to (5');
	     \draw [style=stealth] (8) to (8');
	     \draw [style=stealth] (11) to (11');
	    \draw [style=stealth] (6') to (6);
	    \draw [style=stealth] (9') to (9);
		\draw [style=stealth] (0) to (6);
		\draw [style=stealth] (3) to (9);
		\draw [style=stealth] (1) to (7);
		\draw [style=stealth] (4) to (10);
		\draw [style=stealth] (0) to (1);
		\draw [style=stealth] (1) to (2);
		\draw [style=stealth] (7) to (8);
		\draw [style=stealth] (6) to (7);
		\draw [style=stealth] (0) to (3);
		\draw [style=stealth] (6) to (9);
		\draw [style=stealth] (1) to (4);
		\draw [style=stealth] (8) to (11);
		\draw [style=stealth] (3) to (4);
		\draw [style=stealth] (4) to (5);
		\draw [style=stealth] (2) to (5);
		\draw [style=stealth] (9) to (10);
		\draw [style=stealth] (10) to (11);
	\end{pgfonlayer}
\end{tikzpicture}
\]

We complete the diagram by defining a morphism $\delta \colon {\T'}/{\ker(\ell^\N)} \to {\T}/{\ker(\ell^\N)}$ via the universal property of the cokernel since ${f \circ \pi_\T \circ \kappa'|_{\ker(\ell^\N)} = 1}$, where $\kappa'|_{\ker(\ell^\N)} \colon \ker(\ell^\N) \to \T'$ is the inclusion of the kernel of $\ell^\N$. 
\[ \begin{tikzpicture}[descr/.style={fill=white},baseline=(A.base),yscale=0.8]
\node (A) at (0,0) {$\ker(\ell^\N)$};
\node (B) at (0,2) {$\ker(\ell^\N)$};
\node (C) at (2,2) {$\T'$};
\node (D) at (2,0) {$\T$};
\node (E) at (4,2) {$T' / \ker(\ell^\N)$};
\node (F) at (4,0) {$T / \ker(\ell^\N)$};
\node (O1) at (-2,2) {$1$};
\node (O1') at (-2,0) {$1$};
\node (O2) at (6,2) {$1$};
\node (O2') at (6,0) {$1$};
  \path[-stealth,font=\scriptsize]
  (O1) edge node[above] {$ $} (B)
(E) edge node[above] {$ $} (O2)
(F) edge node[above] {$ $} (O2')
(O1') edge node[above] {$ $} (A)
  (A.east) edge node[below] {$\kappa|_{\ker(\ell^\N)}$} (D.west) 
 (C.south) edge node[right] {$\pi_{\T}$} (D.north) 
 (B.east)  edge node[above] {$\kappa'|_{\ker(\ell^\N)}$} (C.west)

 (C.east)  edge node[above] {$f'$} (E.west)
 (D.east) edge node[below] {$f$} (F.west);
   \path[dashed,-stealth,font=\scriptsize]
  (E.south) edge node[right] {$\delta$} (F.north);
   \draw[commutative diagrams/.cd, ,font=\scriptsize]
(B) edge[commutative diagrams/equal] (A) ;
\end{tikzpicture}\]

\[
\begin{tikzpicture}[yscale=0.8]
	\begin{pgfonlayer}{nodelayer}
		\node [style=green] (0) at (-6.75, 1.25) {$\N$};
		\node [style=green] (0') at (-8.5, 1.25) {$1$};
		\node [style=green] (1) at (-3.75, 1.25) {$\T'$};
		\node [style=green] (2) at (-0.75, 1.25) {$\Q'$};
		\node [style=green] (2') at (1, 1.25) {$1$};
		\node [style=green] (3) at (-6.75, -1.75) {$\N$};
		\node [style=green] (3') at (-8.5, -1.75) {$1$};
		\node [style=green] (4) at (-3.75, -1.75) {$\T$};
		\node [style=green] (5) at (-0.75, -1.75) {$\Q$};
		\node [style=green] (5') at (1, -1.75) {$1$};
		\node [style=green] (6) at (-4.75, 0.25) {$\LN$};
		\node [style=green] (6') at (-6.5, 0.25) {$1$};
		\node [style=green] (7) at (-1.75, 0.25) {$\T' / \ker(\ell^\N)$};
		\node [style=green] (8) at (1.25, 0.25) {$\Q'$};
		\node [style=green] (8') at (3, 0.25) {$1$};
		\node [style=green] (9) at (-4.75, -2.75) {$\LN$};
		\node [style=green] (9') at (-6.5, -2.75) {$1$};
		\node [style=green] (10) at (-1.75, -2.75) {$\T /\ker(\ell^\N)$};
		\node [style=green] (11) at (1.25, -2.75) {$\Q$};
		\node [style=green] (11') at (3, -2.75) {$1$};
		\node [style=none,font=\scriptsize] (12) at (-5.25, 1.5) {$\kappa'$};
		\node [style=none,font=\scriptsize] (13) at (-2, 1.5) {$\pi_{\Q'}$};
		\node [style=none,font=\scriptsize] (16) at (-5.25, -1.5) {$\kappa$};
		\node [style=none,font=\scriptsize] (17) at (-2, -1.5) {$\alpha$};
		\node [style=none,font=\scriptsize] (18) at (-0.5, -0.5) {$g$};
		\node [style=none,font=\scriptsize] (19) at (-3.25, -0.5) {$\pi_{\T}$};
		\node [style=none,font=\scriptsize] (21) at (-3.25, -3) {$j$};
		\node [style=none,font=\scriptsize] (22) at (-3.5, 0.5) {$j'$};
		\node [style=none,font=\scriptsize] (23) at (0, -3) {$p$};
		\node [style=none,font=\scriptsize] (24) at (-0.25, 0.5) {$p'$};
		\node [style=none,font=\scriptsize] (26) at (1.5, -1.25) {$g$};
		\node [style=green] (28) at (-1.75, -0.75) {$\delta$};
		\node [style=none,font=\scriptsize] (29) at (-2.75, 1) {$f'$};
		\node [style=none,font=\scriptsize] (29) at (-2.7, -2) {$f$};
		\node [style=none,font=\scriptsize] (30) at (-5.5, 1) {$l^{\N}$};
		\node [style=none,font=\scriptsize] (31) at (-5.5, -2) {$l^{\N}$};
	\end{pgfonlayer}
		\draw[commutative diagrams/.cd, ,font=\scriptsize]
(0) edge[commutative diagrams/equal] (3);
	\draw[commutative diagrams/.cd, ,font=\scriptsize]
(5) edge[commutative diagrams/equal] (11);
	\draw[commutative diagrams/.cd, ,font=\scriptsize]
(6) edge[commutative diagrams/equal] (9);
	\draw[commutative diagrams/.cd, ,font=\scriptsize]
(2) edge[commutative diagrams/equal] (8);
	\begin{pgfonlayer}{edgelayer}
	    \draw [style=stealth] (0') to (0);
	    \draw [style=stealth] (3') to (3);
	     \draw [style=stealth] (2) to (2');
	     \draw [style=stealth] (5) to (5');
	     \draw [style=stealth] (8) to (8');
	     \draw [style=stealth] (11) to (11');
	    \draw [style=stealth] (6') to (6);
	    \draw [style=stealth] (9') to (9);
		\draw [style=stealth] (0) to (6);
		\draw [style=stealth] (3) to (9);
		\draw [style=stealth] (1) to (7);
		\draw [style=stealth] (4) to (10);
		\draw [style=stealth] (0) to (1);
		\draw [style=stealth] (1) to (2);
		\draw [style=stealth] (7) to (8);
		\draw [style=stealth] (6) to (7);
		\draw [style=stealth] (0) to (3);
		\draw [style=stealth] (6) to (9);
		\draw [style=stealth] (1) to (4);
		\draw [style=stealth] (8) to (11);
		\draw [style=stealth] (3) to (4);
		\draw [style=stealth] (4) to (5);
		\draw [style=stealth] (2) to (5);
		\draw [style=stealth] (9) to (10);
		\draw [style=stealth] (10) to (11);
		\draw [style=dashouille] (7) to (10);
	\end{pgfonlayer}
\end{tikzpicture}
\]
We can check that $\delta$ makes the two front faces commute. Indeed, the right and left faces commute by using the fact that $\ell^\N$ and $f'$ are epimorphisms respectively.

The commutativity of the above diagram and \cref{usefulproposemiab} implies that 

\[ \begin{tikzpicture}[descr/.style={fill=white},baseline=(B.base),yscale = 0.8]
\node (B) at (0,2) {$\LN$};
\node (C) at (2,2) {$\T' / \ker(\ell^\N)$};
\node (E) at (4,2) {$\Q'$};
 \node (O1) at (-2,2) {$1$};
\node (O2) at (6,2) {$1$};
  \path[-stealth,font=\scriptsize]
  (O1) edge node[above] {$ $} (B)
(E) edge node[above] {$ $} (O2)
 (B.east)  edge node[above] {$j'$} (C.west)
 (C.east)  edge node[above] {$p'$} (E.west);
\end{tikzpicture}\]
is the pullback of 
\begin{tikzpicture}[descr/.style={fill=white},baseline=(A.base),scale = 0.8] 
\node (A) at (0,0) {$\T/ \ker(\ell^\N)$};
\node (B) at (2.5,0) {$\Q$};
\node (C) at (-2.5,0) {$\LN$};
\node (D) at (5,0) {$1$};
\node (E) at (-5,0) {$1$};
\path[-stealth,font=\scriptsize]
(C) edge node[above] {$j$} (A)
(A) edge node[above] {$p$} (B)
(B) edge node[above] {$ $} (D)
(E) edge node[above] {$ $} (C)
;
\end{tikzpicture} 
along $g$.
\end{proof}

\begin{remark}
In \cite{FS}, the construction of the fiberwise localization in the category of groups was functorial, therefore from the morphism $\T' \to \T$ between the pullback sequence and the sequence itself we have directly a morphism between the fiberwise localization of the pullback sequence and the fiberwise localization of the original sequence. In other words the map $\delta$ comes for free in contrast to the category of crossed modules where we have to build the map $\delta$ explicitly. 
\end{remark}


\section{Conditional flatness and admissibility}\label{section CF and Ad}

In this section, we develop a simpler characterisation of conditional flatness, thanks to the results of the previous section. We introduce the notion of admissibility for the class of regular epimorphisms and show that it is equivalent to conditional flatness. 
With this equivalence, we can easily establish conditional flatness for a given localization functor. We observe that some properties of localization functors, such as right-exactness, imply directly admissibility for the class of regular epimorphism.\newline

The first step allows us to restrict the definition of conditional flatness (\cref{defcondflat}) to fiberwise localizations of $\L$-flat exact sequences (\cref{1ststep}). More precisely, we show that the pullback of an $\L$-flat exact sequence is $\L$-flat if and only if the pullback of its fiberwise localization is so.

\begin{lemma}\label{1ststep}
 Let $\L$ be a regular-epi localization functor. Then $\L$ is conditionally flat if and only if for any  $\L$-flat exact sequence \begin{tikzpicture}[descr/.style={fill=white},baseline=(A.base),scale = 0.8] 
\node (A) at (0,0) {$\T$};
\node (B) at (2.5,0) {$\Q$};
\node (C) at (-2.5,0) {$\N$};
\node (D) at (5,0) {$1$};
\node (E) at (-5,0) {$1$};
\path[-stealth,font=\scriptsize]
(C) edge node[above] {$\kappa$} (A)
(A) edge node[above] {$\alpha$} (B)
(B) edge node[above] {$ $} (D)
(E) edge node[above] {$ $} (C);
\end{tikzpicture} with $\N$ an $\L$-local crossed module, the pullback sequence along any morphism $\Q' \rightarrow \Q$ is $\L$-flat. 
\end{lemma}

\begin{proof}
This is clear since $f'$ and $\ell^\N$ are $\L$-equivalences in this diagram: \[ \begin{tikzpicture}[descr/.style={fill=white},baseline=(A.base)yscale=0.8]
\node (A) at (0,0) {$\LN$};
\node (B) at (0,2) {$\N$};
\node (C) at (2,2) {$\T'$};
\node (D) at (2,0) {${\T'}/{\ker(\ell^\N)}$};
\node (E) at (4,2) {$\Q'$};
\node (F) at (4,0) {$\Q'$};
\node (O1) at (-2,2) {$1$};
\node (O1') at (-2,0) {$1$};
\node (O2) at (6,2) {$1$};
\node (O2') at (6,0) {$1$};
  \path[-stealth,font=\scriptsize]
  (O1) edge node[above] {$ $} (B)
(E) edge node[above] {$ $} (O2)
(F) edge node[above] {$ $} (O2')
(O1') edge node[above] {$ $} (A)
  (A.east) edge node[below] {$j'$} (D.west) 
 (C.south) edge node[right] {$f'$} (D.north) 
 (B.south) edge node[left] {$\ell_\N$} (A.north) 
 (B.east)  edge node[above] {$\kappa'$} (C.west)
 (C.east)  edge node[above] {$\pi_{\Q'}$} (E.west)
 (D.east) edge node[below] {$p'$} (F.west);
 \draw[commutative diagrams/.cd, ,font=\scriptsize]
(E) edge[commutative diagrams/equal] (F);
\end{tikzpicture}
\]
The top row is thus $\L$-flat if and only if so is the bottom row and we conclude by \cref{prop:pullbackfiber=fiberpullback}.
\end{proof}

The previous lemma allows us to follow the approach introduced in \cite{FS}. For the sake of completeness, we give an explicit proof of the following results even if the arguments are similar to 
the group theoretical ones.  

\begin{prop}\label{2ndstep}
 Let $\L$ be a regular-epi localization functor. Then $\L$ is conditionally flat if and only if the pullback of any exact sequence of $\L$-local objects is $\L$-flat. 
\end{prop}

\begin{proof}
By the previous lemma it is sufficient to consider exact sequence with an $\L$-local kernel $\L\N$. Consider thus an $\L$-flat exact sequence \begin{tikzpicture}[descr/.style={fill=white},scale= 0.8, baseline=(A.base)] 
\node (A) at (0,0) {$\T$};
\node (B) at (2.5,0) {$\Q$};
\node (C) at (-2.5,0) {$\L\N$};
\node (O1) at (-5,0) {$1$};
\node (O2) at (5,0) {$1$};
\path[-stealth,font=\scriptsize]
(C.east) edge node[above] {$j$} (A.west)
(O1) edge node[above] {$ $} (C)
(B) edge node[above] {$ $} (O2)
(A) edge node[above] {$p$} (B);
\end{tikzpicture}. 
We build the following diagram where $g \colon \Q' \to \Q$ is any morphism of crossed modules and (1) is a pullback.  
\begin{center}
\begin{tikzpicture}[descr/.style={fill=white},yscale=1,baseline=(A.base)]
\node (A') at (0,-2) {$\LN$};
\node (A) at (0,0) {$\LN$};
\node (B) at (0,2) {$\LN$};
\node (C) at (3,2) {$\T'$};
\node (X) at (4.5,1) {$(1)$};
\node (X) at (4.5,-1) {$(2)$};
\node (D) at (3,0) {${\T}$};
\node (D') at (3,-2) {$\mathsf{L}\T$};
\node (E) at (6,2) {$\Q'$};
\node (F) at (6,0) {$\Q$};
\node (F') at (6,-2) {$\LQ$};
\node (O1) at (-2,2) {$1$};
\node (O1') at (-2,0) {$1$};
\node (O2) at (8,2) {$1$};
\node (O2') at (8,0) {$1$};
\node (O2'') at (8,-2) {$1$};
\node (O1'') at (-2,-2) {$1$};
  \path[-stealth,font=\scriptsize]
  (O1) edge node[above] {$ $} (B)
(E) edge node[above] {$ $} (O2)
(F) edge node[above] {$ $} (O2')
(F') edge node[above] {$ $} (O2'')
(O1'') edge node[above] {$ $} (A')
(O1') edge node[above] {$ $} (A)
  (A.east) edge node[below] {$j$} (D.west) 
    (A'.east) edge node[below] {$L(j)$} (D'.west) 
  (D) edge node[descr] {$\ell^{\T}$} (D')
   (F) edge node[descr] {$\ell^{\Q}$} (F')
 (C.south) edge node[descr] {$\pi_{\T}$} (D.north) 
 (B.east)  edge node[above] {$j'$} (C.west)
 (E.south) edge node[right] {$g$} (F.north)
 (C.east)  edge node[above] {$\pi_{\Q'}$} (E.west)
  (D'.east) edge node[below] {$L(p)$} (F'.west)
 (D.east) edge node[below] {$p$} (F.west);
 \draw[commutative diagrams/.cd,font=\scriptsize]
(A) edge[commutative diagrams/equal] (A')
(B) edge[commutative diagrams/equal] (A);
\end{tikzpicture}

\end{center}
We observe that since each row is exact, $(2)$ is a pullback by \cref{usefulproposemiab}, and then $(1)+(2)$ is also a pullback. Hence, 
the top row is the pullback of the bottom exact sequence of $\L$-local objects along the map $\ell^\Q \circ g$, which shows the claim.
\end{proof}

 \begin{definition}
 A localization functor $\L$ is said to be \emph{admissible} for the class of regular epimorphisms if it preserves any pullback of the form \[
\begin{tikzpicture}[descr/.style={fill=white},yscale=0.8]
\node (A) at (0,0) {$\LT$};
\node (B) at (0,2) {$\T'$};
\node (C) at (2,2) {$\Q$};
\node (D) at (2,0) {$\LQ$};
  \path[-stealth,font=\scriptsize]
 (B.south) edge node[left] {$\pi_\LT$} (A.north) 
 (C.south) edge node[right] {$\ell^\Q$} (D.north) 
  (B.east) edge node[above] {$\pi_{\Q}$} (C.west) 
 (A.east) edge node[below] {$\alpha$} (D.west);
\end{tikzpicture} 
\] where $\alpha$ is a regular epimorphism.
\end{definition}

\begin{theorem}\label{finalstep}
Let $\L$ be a regular-epi localization functor. Then the following statements are equivalent
\begin{enumerate}
    \item $\L$ is conditionally flat;
    \item $\L$ is admissible for the class of regular epimorphisms.
\end{enumerate}
\end{theorem}

\begin{proof}
The implication $(1) \Rightarrow (2)$ is trivial, so let us prove $(2) \Rightarrow (1)$. Consider any exact sequence of $\L$-local objects 
\[ 
\begin{tikzpicture}[descr/.style={fill=white},baseline=(A.base),
xscale=0.8] 
\node (A) at (0,0) {$\LT$};
\node (B) at (2.5,0) {$\LQ$};
\node (C) at (-2.5,0) {$\LN$};
\node (O1) at (-4.5,0) {$1$};
\node (O2) at (4.5,0) {$1$};
\path[-stealth,font=\scriptsize]
(C.east) edge node[above] {$ $} (A.west)
;
\path[-stealth,font=\scriptsize]
(O1) edge node[above] {$ $} (C)
(B) edge node[above] {$ $} (O2)
(A) edge node[above] {$\alpha $} (B);
\end{tikzpicture} 
\] 
and any morphism $g \colon \A \to \LQ$. By \cref{2ndstep} conditional flatness is established if we prove that the pullback of the exact sequence along $g$ is $\L$-flat.
Let us first observe that this morphism $g$ factors through $\LA$ via the universal property of the localization:
\[
\begin{tikzpicture}[descr/.style={fill=white},baseline=(A.base),
scale=0.8] 
\node (A) at (0,0) {$\A$};
\node (B) at (2.5,0) {$\LA$};
\node (A') at (2.5,-2) {$\LQ$};
\path[-stealth,font=\scriptsize]
(A) edge node[above] {$\ell^\A$} (B)
(A) edge node[descr] {$ g $} (A');
\path[dashed,-stealth,font=\scriptsize]
(B) edge node[right] {$ \tilde{g} $} (A');
\end{tikzpicture} 
\]
Hence, we can first construct the pullback of \begin{tikzpicture}[descr/.style={fill=white},baseline=(A.base),
xscale=0.8] 
\node (A) at (0,0) {$\LT$};
\node (B) at (2.5,0) {$\LQ$};
\node (C) at (-2.5,0) {$\LN$};
\node (O1) at (-4.5,0) {$1$};
\node (O2) at (4.5,0) {$1$};
\path[-stealth,font=\scriptsize]
(C.east) edge node[above] {$ $} (A.west)
;
\path[-stealth,font=\scriptsize]
(O1) edge node[above] {$ $} (C)
(B) edge node[above] {$ $} (O2)
(A) edge node[above] {$\alpha $} (B);
\end{tikzpicture} 
along $\tilde{g}$ and then pullback the resulting sequence along $\ell^\A$:
\[ 
\begin{tikzpicture}[descr/.style={fill=white},scale=1,baseline=(A.base),yscale=1]
\node (A') at (0,-2) {$\LN$};
\node (A) at (0,0) {$\LN$};
\node (B) at (0,2) {$\LN$};
\node (C) at (2,2) {$\T''$};
\node (D) at (2,0) {$\T'$};
\node (D') at (2,-2) {$\LT$};
\node (E) at (4,2) {$\A$};
\node (F) at (4,0) {$\LA$};
\node (F') at (4,-2) {$\LQ$};
\node (O1) at (-2,2) {$1$};
\node (O1') at (-2,0) {$1$};
\node (O2) at (6,2) {$1$};
\node (O2') at (6,0) {$1$};
\node (O2'') at (6,-2) {$1$};
\node (O1'') at (-2,-2) {$1$};
  \path[-stealth,font=\scriptsize]
  (O1) edge node[above] {$ $} (B)
(E) edge node[above] {$ $} (O2)
(F) edge node[above] {$ $} (O2')
(F') edge node[above] {$ $} (O2'')
(O1'') edge node[above] {$ $} (A')
(O1') edge node[above] {$ $} (A);
\draw[-stealth,>=latex] 
(E.south east) to[bend left=25]node[descr] {$g$} (F'.north east);
  \path[-stealth]
  (A.east) edge node[below] {$ $} (D.west) 
    (A'.east) edge node[below] {$ $} (D'.west) 
  (D) edge node[descr] {$\pi_\LT$} (D')
   (F) edge node[right] {$\tilde{g}$} (F')
 (C.south) edge node[right] {$ $} (D.north) 
 (B.east)  edge node[above] {$ $} (C.west)
 (E.south) edge node[right] {$\ell^\A$} (F.north)
 (C.east)  edge node[above] {$\pi_{\A}$} (E.west)
  (D'.east) edge node[below] {$\alpha$} (F'.west)
 (D.east) edge node[below] {$\pi_\LA$} (F.west);
 \draw[commutative diagrams/.cd, ,font=\scriptsize]
(A) edge[commutative diagrams/equal] (A')
(B) edge[commutative diagrams/equal] (A);
\end{tikzpicture}
\]
Since the category of $\L$-local objects is closed under pullbacks, $\T'$ is $\L$-local and we can apply condition (2) to conclude that the upper row is $\L$-flat. This observation implies that the pullback of
$\begin{tikzpicture}[descr/.style={fill=white},baseline=(A.base),
scale=0.8] 
\node (A) at (0,0) {$\LT$};
\node (B) at (2.5,0) {$\LQ$};
\node (C) at (-2.5,0) {$\LN$};
\node (O1) at (-4.5,0) {$1$};
\node (O2) at (4.5,0) {$1$};
\path[-stealth,font=\scriptsize]
(C.east) edge node[above] {$ $} (A.west)
;
\path[-stealth,font=\scriptsize]
(O1) edge node[above] {$ $} (C)
(B) edge node[above] {$ $} (O2)
(A) edge node[above] {$\alpha $} (B);
\end{tikzpicture}$ along $g$ is an $\L$-flat sequence as desired.
\end{proof}

The above theorem gives an easier characterisation of conditionally flatness in the category of crossed modules. It will be useful in rest of the article. 

\begin{remark}
Admissibility for the class of regulars epimorphisms in the context of semi-abelian  categories is studied in \cite{GranScherer}. Similar results are proven for functors of localizations that admit a functorial fiberwise localization. Note that their result does not imply \cref{finalstep} since localization functors of crossed modules do not admit functorial fiberwise localizations in general. However, the implication ``$(1)$ implies $(2)$", in \cref{finalstep}, holds even for not necessarily regular-epi localization functors.
\end{remark}



\begin{prop}
If $\L \colon \sf XMod \to \sf XMod$ is a localization functor that is right exact in $\sf XMod$, then $\L$ is admissible for the class of regular epimorphisms.
\end{prop} 

\begin{proof}
Let us consider the following pullback of an $\L$-flat exact sequence of crossed modules along a morphism $g \colon \Q' \to \Q$.
\[
\begin{tikzpicture}[descr/.style={fill=white},baseline=(A.base),scale=0.8] 
\node (A) at (0,0) {$\T'$};
\node (B) at (2.5,0) {$\Q'$};
\node (C) at (-2.5,0) {$\N$};
\node (A') at (0,-2) {$\T$};
\node (B') at (2.5,-2) {$\Q$};
\node (C') at (-2.5,-2) {$\N$};
\node (O1) at (-4.5,0) {$1$};
\node (O1') at (-4.5,-2) {$1$};
\node (O2) at (4.5,0) {$1$};
\node (O2') at (4.5,-2) {$1$};
\node (X) at (1.25,-1) {$(1)$};
\draw[commutative diagrams/.cd, ,font=\scriptsize]
(C) edge[commutative diagrams/equal] (C');
\path[-stealth,font=\scriptsize]
(O1) edge node[above] {$ $} (C)
(B) edge node[above] {$ $} (O2)
(B') edge node[above] {$ $} (O2')
(O1') edge node[above] {$ $} (C')
(B.south) edge node[right] {$ g $}  (B'.north)

(A.south) edge node[left] {$ \pi_\T $} (A'.north)
(C'.east) edge node[above] {$ \kappa $} (A'.west)
(A') edge node [below] {$f$} (B') ([yshift=2pt]A'.east)
(C.east) edge node[above] {$ \kappa' $} (A.west)
(A) edge node[above] {$\pi_{\Q'}$} (B);
\end{tikzpicture} \]
By applying $\L$ to this diagram, we obtain (since $\L$ is right exact) the following diagram

\[
\begin{tikzpicture}[descr/.style={fill=white},baseline=(A.base),scale=0.8] 
\node (A) at (0,0) {$\LT'$};
\node (B) at (2.5,0) {$\LQ'$};
\node (C) at (-2.5,0) {$\LN$};
\node (A') at (0,-2) {$\LT$};
\node (B') at (2.5,-2) {$\LQ$};
\node (C') at (-2.5,-2) {$\LN$};
\node (O1') at (-4.5,-2) {$1$};
\node (O2) at (4.5,0) {$1$};
\node (O2') at (4.5,-2) {$1$};
\draw[commutative diagrams/.cd, ,font=\scriptsize]
(C) edge[commutative diagrams/equal] (C');
\path[-stealth,font=\scriptsize]
(B) edge node[above] {$ $} (O2)
(B') edge node[above] {$ $} (O2')
(O1') edge node[above] {$ $} (C')
(B.south) edge node[right] {$ \L(g) $}  (B'.north)
(A.south) edge node[left] {$ \L(\pi_\T )$} (A'.north)
(C'.east) edge node[below] {$\L(\kappa) $} (A'.west)
(A') edge node [below] {$\L(f)$} (B') ([yshift=2pt]A'.east)
(C.east) edge node[above] {$\L(\kappa') $} (A.west)
(A) edge node[above] {$\L(\pi_{\Q'})$} (B);
\end{tikzpicture} \]
Since $\L(\kappa) = \L(\pi_\T) \circ \L(\kappa')$ is a (normal) monomorphism, we conclude that $\L(\kappa')$ is a monomorphism. 
 Normality follows then by right-exactness and we conclude by \cref{finalstep}.
\end{proof}
Note that this proof holds in any semi-abelian category.

\begin{corollary}
The functor of abelianization $\sf Ab \colon \sf XMod \to \sf XMod$ is admissible for the class of regular epimorphisms.
\end{corollary}

\begin{proof}
The functor of abelianization $\sf Ab \colon \sf XMod \to \sf XMod$ is right exact. Since the exactness can be shown component-wise, the result follows.
\end{proof}

Sometimes it is handy to rely on our group theoretical knowledge
to construct simple examples of localization functors and how they
behave on crossed modules. The proof of the following proposition is based on a counter-example coming from groups via the functor $\sf X$ defined in \cref{lemma_lefttruncation}.

\begin{prop}
There are regular-epi localization functors $\L \colon \sf XMod \to \sf XMod$ that are not admissible for the class of regular epimorphisms.
\end{prop}
\begin{proof}

We export via $\sf X \colon Grp \to XMod$ the example in \cite[Theorem~5.1]{FS} of a localization functor in groups that is not admissible for the class of regular epimorphisms. 

Let $\L_\phi$ be the localization functor induced by the projection $\phi \colon C_4 \to C_2 $, where $C_n$ denotes a cyclic group of order $n$.
It gives rise to a localization functor  $\sf L_{X\phi} \colon \sf XMod \to \sf XMod$. In particular, if we apply $\sf X$ to the extension of $\L_\phi$-local groups considered in \cite{FS}, we obtain an exact sequence of $\L_{\sf X \phi}$-local crossed modules:

\[ \begin{tikzpicture}[descr/.style={fill=white},baseline=(A.base),
xscale=1] 
\node (A) at (0,0) {$(1, \mathbb{Z})$};
\node (B) at (2.5,0)
{$(1,C_2)$};
\node (C) at (-2.5,0) 
{$(1, \mathbb{Z})$};
\node (O1) at (-4.5,0) {$1$};
\node (O2) at (4.5,0) {$1$};
\path[-stealth,font=\scriptsize]
(C.east) edge node[above] {$ $} (A.west)
;
\path[-stealth,font=\scriptsize]
(O1) edge node[above] {$ $} (C)
(B) edge node[above] {$ $} (O2)
(A) edge node[above] {$ $} (B);
\end{tikzpicture} \]
If we pullback along the morphism of crossed modules $\sf X \phi$, we obtain the following exact sequence 
\[ \begin{tikzpicture}[descr/.style={fill=white},baseline=(A.base),
xscale=1] 
\node (A) at (0,0) {$(1, \mathbb{Z} \times C_2) $};
\node (B) at (2.5,0)
{$(1,C_4)$};
\node (C) at (-2.5,0) 
{$(1, \mathbb{Z})$};
\node (O1) at (-4.5,0) {$1$};
\node (O2) at (4.5,0) {$1$};
\path[-stealth,font=\scriptsize]
(C.east) edge node[above] {$ $} (A.west)
;
\path[-stealth,font=\scriptsize]
(O1) edge node[above] {$ $} (C)
(B) edge node[above] {$ $} (O2)
(A) edge node[above] {$ $} (B);
\end{tikzpicture} \]

We conclude from \cite[Lemma~1.4]{MSS} that this exact sequence is not $\L_{\sf X \phi}$-flat. Indeed, if it was the case we would have a contradiction with the group theoretical observation in \cite{FS}.
\end{proof}

\section{Admissibility and acyclicity }

In the categories of groups and topological spaces, the localization functor $\L$ is a nullification functor if and only if the kernels of the localization morphisms are $\L$-acyclic (which means that $\L\ker(\ell^\M)$ is trivial for any $\M \in \sf XMod$). This characterisation implies in particular that any nullification functor is admissible for the class of regular epimorphisms. 
It is interesting to notice that even if nullification functors of crossed modules do not have acyclic kernels, we have a similar result in $\sf XMod$: the $\sf L$-acyclicity of the kernels of localization morphisms implies the admissibility.

\begin{prop}
\label{prop:acyclicimpliesadmissible}
Let $\L\colon \sf XMod \to XMod$ be a regular-epi localization functor such that $\ker(\ell^\M \colon \M \to \L\M )$ is $\L$-acyclic for any $\M \in \sf XMod$. Then $\L$ is admissible for the class of regular epimorphisms.
\end{prop}
  
\begin{proof}
Consider the pullback of 
$\begin{tikzpicture}[descr/.style={fill=white},baseline=(A.base),
xscale=0.8] 
\node (A) at (0,0) {$\LT$};
\node (B) at (2.5,0) {$\LQ$};
\node (C) at (-2.5,0) {$\LN$};
\node (O1) at (-4.5,0) {$1$};
\node (O2) at (4.5,0) {$1$};
\path[-stealth,font=\scriptsize]
(C.east) edge node[above] {$\kappa $} (A.west)
;
\path[-stealth,font=\scriptsize]
(O1) edge node[above] {$ $} (C)
(B) edge node[above] {$ $} (O2)
(A) edge node[above] {$ f $} (B);
\end{tikzpicture}$ 
along $\ell^\Q \colon \Q \to \LQ$:

\[ 
\begin{tikzpicture}[descr/.style={fill=white},baseline=(A.base),yscale=0.8]
\node (A) at (0,0) {$\LN$};
\node (B) at (0,2) {$\LN$};
\node (C) at (2,2) {$\T'$};
\node (D) at (2,0) {$\LT$};
\node (E) at (4,2) {$\Q$};
\node (F) at (4,0) {$\LQ$};
\draw[commutative diagrams/.cd, ,font=\scriptsize]
(B) edge[commutative diagrams/equal] (A);
\node (O1) at (-2,2) {$1$};
\node (O1') at (-2,0) {$1$};
\node (O2) at (6,2) {$1$};
\node (O2') at (6,0) {$1$};
  \path[-stealth,font=\scriptsize]
  (O1) edge node[above] {$ $} (B)
(E) edge node[above] {$ $} (O2)
(F) edge node[above] {$ $} (O2')
(O1') edge node[above] {$ $} (A)
  (A.east) edge node[below] {$\kappa$} (D.west) 
 (C.south) edge node[right] {$\pi_{\LT}$} (D.north) 
 (B.east)  edge node[above] {$\kappa'$} (C.west)
 (E.south) edge node[right] {$\ell^\Q$} (F.north)
 (C.east)  edge node[above] {$\pi_Q$} (E.west)
 (D.east) edge node[below] {$f$} (F.west);
\end{tikzpicture}
\]

We need to prove that $\pi_{\LT}$ is an $\L$-equivalence. Since $\sf XMod$ is a pointed protomodular category and $\ell^\Q$ is a regular epi by assumption, we know that $\pi_{\LT}$ is the cokernel of $\ker(\ell^\Q) \cong \ker(\pi_{\LT}) \to \T'$. Let $\Y$ be a local object, for any $g \colon \T' \to \Y$ we have the following diagram:
\[
\begin{tikzpicture}[descr/.style={fill=white},baseline=(A.base),scale=0.8] 
\node (O) at (-2.5,0) {$\ker(\ell^\Q)$};
\node (O') at (-2.5,-2) {$\L\ker(\ell^\Q)=1$};
\node (A) at (0,0) {$\T'$};
\node (B) at (2.5,0) {$\LT$};
\node (A') at (0,-2) {$\Y$};
\path[-stealth,font=\scriptsize]
(O) edge (A)
(O) edge (O')
(O') edge node[above] {$ g'$} (A')
(A) edge node[above] {$\pi_{\LT}$} (B)
(A) edge node[descr] {$ g $} (A');
\path[dashed,-stealth,font=\scriptsize]
(B) edge node[right] {$ \tilde{g} $} (A');
\end{tikzpicture} 
\]
By the universal property of the localization there exists $g' \colon 1 \to \Y$ that makes the left square commute. Hence, by the universal property of the cokernel there exists a unique $\tilde{g} \colon \LT \to \Y$ such that the triangle commutes and we  conclude that $\pi_{\LT}$ is an $\L$-equivalence.
\end{proof}

However, localization functors of crossed modules do not behave like localization functors of groups. As explained above, in the category of groups (but also of topological spaces), the kernels of the localization morphisms are $\L$-acyclic 
if and only if $\L$ is a nullification functor \cite{FS}. In the context of crossed modules, we do not have such a characterization of nullification functors.

\begin{remark}
\label{nullificationnonacyclic}
We know by \cite[Proposition 4.6]{MSS} that there are nullification functors, for example $\P_{\sf X\mathbb{Z}}$ defined in \cref{example_PZ}, such that the kernels of their localization morphisms are not acyclic in general. Still, in the next proposition, we prove that if the kernel of the localization morphism is $\L$-acyclic, as in \cref{prop:acyclicimpliesadmissible}, then the localization functor is a nullification.

\end{remark}

 The cardinal in the next proof is chosen exactly as in Bousfield's \cite[Theorem~4.4]{MR1481817} for spaces.

\begin{prop}
\label{prop:acyclicimpliesnullification}
    Let $f \colon \B \to \C$ be a morphism of crossed modules and $\L_f \colon \sf XMod \to XMod$ be a regular-epi localization functor. If the kernels of the localization morphisms are $\L_f$-acyclic, then $\L_f$ is a nullifcation functor. 
\end{prop}

\begin{proof}
Our strategy is to construct a crossed module $\A$ such that we can compare the functor $\L_f$ with the nullification functor $\P_\A$ (\cref{null}) via a natural transformation $\psi$. 
We choose $\kappa$ to be the first infinite ordinal greater than the number of chosen generators of $\B$ and $\C$, i.e., generators of the groups $B_1$, $B_2$, $C_1$ and $C_2$.
 We construct the crossed module $\A := \coprod \A_{\alpha} $, where $\A_{\alpha}$ are all the $\L_f$-acyclic crossed modules with less than $2^\kappa$ generators, see \cite[Theorem 4.4]{MR1481817}. 

The first step of this proof is to show that if a crossed module $\X$ is $\L_f$-local then it is $\A$-local. Let $\phi $ be a morphism in $Hom(\A,\X)$ and construct by naturality the following commutative diagram
\[
\begin{tikzpicture}[descr/.style={fill=white},scale = 0.8]
\node (A) at (0,0) {$1=\L_f\A$};
\node (B) at (0,2) {$\A$};
\node (C) at (3,2) {$\X$};
\node (D) at (3,0) {$\L_f\X$};
  \path[-stealth,font=\scriptsize]
  (A.east) edge node[below] {$\L_f\phi$} (D.west) 
 (C.south) edge node[right] {$\cong$} (D.north) 
 (B.south) edge node[right] {$ $} (A.north) 
 (B.east)  edge node[above] {$\phi$} (C.west);
\end{tikzpicture}
\]
By hypothesis, we have an isomorphism between $\X$ and $\L_f\X$ and by construction of $\A$, we obtain $\L_f\A = 1$. Therefore, $\phi$ factors through the zero object and hence $Hom(\A,\X) = 1$,
which is equivalent to say that $\X$ is $\A$-local. Now consider the $\P_\A$-equivalence $p^\T \colon \T \to \P_\A \T$ and the $\L_f$-local object $\L_f\T$. By the above observation, we have that $\L_f\T$ is $\A$-local and by the universal property we have the desired morphism $\psi_\T$
\[
\begin{tikzpicture}[descr/.style={fill=white},baseline=(A.base),
scale=0.8] 
\node (A) at (0,0) {$\T$};
\node (B) at (2.5,0) {$\P_\A\T$};
\node (A') at (2.5,-2) {$\L_f\T$};
\path[-stealth,font=\scriptsize]
(A) edge node[above] {$p^\T$} (B)
(A) edge node[descr] {$ \ell^\T $} (A');
\path[dashed,-stealth,font=\scriptsize]
(B) edge node[right] {$ \psi_\T $} (A');
\end{tikzpicture} 
\]
We construct next the fiberwise $\A$-nullification of the following exact sequence 
\[
\begin{tikzpicture}[descr/.style={fill=white},baseline=(A.base),scale=0.8] 
\node (A) at (0,0) {$\T$};
\node (B) at (2.5,0) {$\L_f\T$};
\node (C) at (-2.5,0) {$\ker(\ell^\T)$};
\node (O1) at (-4.5,0) {$1$};
\node (O2) at (4.5,0) {$1$};
\path[-stealth,font=\scriptsize]
(O1) edge node[above] {$ $} (C)
(B) edge node[above] {$ $} (O2)
(C.east) edge node[above] {$ $} (A.west)
(A) edge node[above] {$\ell^\T$} (B);
\end{tikzpicture} 
\] 
By assumption $\ker(\ell^\T)$ is $\L_f$-acyclic, hence also $\P_\A$-acyclic by design. This implies that $\ker\Big(p^\T \colon \ker(\ell^\T) \to \mathsf{P_A}\ker(\ell^\T) \Big) $ is equal to  $\ker(\ell^\T)$. Hence, the exact sequence satisfies condition \eqref{Conditionnormalite} of \cref{thmfiberwiselocalization} and we obtain the following fiberwise nullification
\[ 
\begin{tikzpicture}[descr/.style={fill=white},baseline=(A.base),scale=0.8] 
\node (A) at (0,0) {$\T$};
\node (B) at (2.5,0) {$\L_f\T$};
\node (C) at (-2.5,0) {$\ker(\ell^\T)$};
\node (A') at (0,-2) {$\T/\ker(\ell^\T)$};
\node (B') at (2.5,-2) {$\L_f\T$};
\node (C') at (-2.5,-2) {$1$};
\node (O1) at (-4.5,0) {$1$};
\node (O1') at (-4.5,-2) {$1$};
\node (O2) at (4.5,0) {$1$};
\node (O2') at (4.5,-2) {$1$};
\draw[commutative diagrams/.cd, ,font=\scriptsize]
(B) edge[commutative diagrams/equal] (B');
\path[-stealth,font=\scriptsize]
(O1) edge node[above] {$ $} (C)
(B) edge node[above] {$ $} (O2)
(B') edge node[above] {$ $} (O2')
(O1') edge node[above] {$ $} (C')
(C.south) edge node[right] {$ p^\T $}  (C'.north)
(A.south) edge node[left] {$ f $} (A'.north)
(C'.east) edge node[above] {$ $} (A'.west)
(A') edge node [above] {$ \cong $} (B') ([yshift=2pt]A'.east)
(C.east) edge node[above] {$ $} (A.west)
(A) edge node[above] {$\ell^\T$} (B);
\end{tikzpicture} 
\]
Since $f$ is a $\sf P_A$-equivalence, so is  $\ell^\T$. Hence, we obtain a morphism $\varphi^\T$ in the following commutative diagram:
\[
\begin{tikzpicture}[descr/.style={fill=white},baseline=(A.base),
scale=0.9] 
\node (A) at (0,0) {$\T$};
\node (B) at (2.5,0) {$\L_f\T$};
\node (A') at (2.5,-2) {$\P_\A\T$};
\draw[-stealth,font=\scriptsize,>=latex] (A') to[bend left=15]node[left] {$ \; \psi^\T $} (B);
\path[-stealth,font=\scriptsize]
(A) edge node[above] {$\ell^\T$} (B)
(A) edge node[descr] {$ p^\T $} (A');
\draw[dashed,-stealth,font=\scriptsize,>=latex] (B) to[bend left=15]node[right] {$ \; \varphi^\T $} (A');
\end{tikzpicture} 
\]
By universal property, we can conclude that the two compositions of $\psi^\T$ and $\varphi^\T$ are isomorphic to identities so that $\L_f\T \cong \P_\A\T$.
A similar argument shows the naturality of $\psi$ and $\varphi$ and therefore $\L_f$ is a nullification functor, namely $\P_\A$.
\end{proof}


\section{Nullification functors and admissibility}
\label{sec:nullification and admissibility}

 In the category of groups, the fact that kernels of localization morphisms are $\L$-acyclic was fundamental to prove that nullification functors are admissible for the class of regular epimorphisms.
  This fact is not true in general for nullification functors in the category of crossed modules as shown in \cite[Proposition~4.6]{MSS}, it is thus natural to ask whether nullification functors are admissible. We provide an affirmative answer in this final section, but let us first prove that our counter-example $\P_{\X\mathbb{Z}}$ is admissible.

\begin{prop}
    The nullification functor $\P_{\X\mathbb{Z}}$ is admissible for the class of regular epimorphisms.
\end{prop}

\begin{proof}
Theorem 5.1 in \cite{GranScherer} implies that $\P_{\X\mathbb{Z}}$ is admissible provided that the reflective category of  $\P_{\X\mathbb{Z}}$-local objects is a Birkhoff subcategory, i.e., it is closed under regular quotients and subobjects. Here $\P_{\X\mathbb{Z}}$-local objects are crossed modules of the form $A \to 1$ where $A$ is any abelian group and the connecting homomorphism is the trivial homomorphism. Therefore it is clearly closed under subobjects. Moreover, the quotient of $A \to 1$ by a normal subcrossed modules $N\to 1$ is the crossed module $A/N \to 1$ that is $\P_{\X\mathbb{Z}}$-local.
\end{proof}

The remaining part of the section is devoted to the proof that \emph{all} nullification functors are admissible for the class of regular epmorphisms. Consider a nullification functor $\sf P_A$ where $\A = (A_1, A_2, \partial) $ is a crossed module. To show the admissibility, it is enough to prove that the pullback of an exact sequence of $\P_\A$-local crossed modules along the coaugmentation map is $\sf P_A$-flat, in other words that the map $f$ in the following commutative diagram of crossed modules is a $\sf P_A$-equivalence
\[
\begin{tikzpicture}[descr/.style={fill=white},baseline=(A.base),scale=0.8] 
\node (A) at (0,0) {$\W$};
\node (B) at (2.5,0) {$\Q$};
\node (C) at (-2.5,0) {$\PAN$};
\node (A') at (0,-2) {$\PAT$};
\node (B') at (2.5,-2) {$\PAQ$};
\node (C') at (-2.5,-2) {$\PAN$};
\node (O1) at (-4.5,0) {$1$};
\node (O1') at (-4.5,-2) {$1$};
\node (O2) at (4.5,0) {$1$};
\node (O2') at (4.5,-2) {$1$};
\node (X) at (1.25,-1) {$(1)$};
\draw[commutative diagrams/.cd, ,font=\scriptsize]
(C) edge[commutative diagrams/equal] (C');
\path[-stealth,font=\scriptsize]
(O1) edge node[above] {$ $} (C)
(B) edge node[above] {$ $} (O2)
(B') edge node[above] {$ $} (O2')
(O1') edge node[above] {$ $} (C')
(B.south) edge node[right] {$ p^\Q $}  (B'.north)
(A.south) edge node[left] {$ f $} (A'.north)
(C'.east) edge node[above] {$ $} (A'.west)
(A') edge node [below] {$h$} (B') ([yshift=2pt]A'.east)
(C.east) edge node[above] {$ $} (A.west)
(A) edge node[above] {$g$} (B);
\end{tikzpicture} \] where $(1)$ is a pullback and $g$ and $h$ are regular epimorphisms. 
 To do so we follow step by step the inductive construction of $\sf P_A Q = colim Q_\beta$ as presented in \cite[Proposition~2.8]{MSS}, see also \cref{propnullify}. For each successor ordinal $\beta+1$ we obtain $\sf Q_{\beta +1}$ from $\sf Q_\beta$ by killing all morphisms out of $\sf A$ so let us start with the construction of $\sf Q_1$ from $\sf Q_0 = Q$.


 \begin{remark} \label{rem:ideabehindnextproof}
 Let $\varphi : \A \rightarrow \Q$ be a morphism of crossed modules. The crossed module $\Q_1$ is the quotient of $\Q$ by the normal closure $\K_\Q$ in $\Q$ of the image of
 \[
 ev\colon \coprod_{\varphi \in Hom(\A,\Q)}\A = \M \longrightarrow \Q
 \]
 which is defined by $\varphi$ on the copy of $\A$ indexed by $\varphi$. The idea behind the construction we perform next is that we do not need to kill all morphisms from $\A$ to the extension $\W$ in order to construct its nullification $\sf P_A \W$, it is sufficient to take care of those factoring through $\Q$. Beware that given an extension $\N \rightarrow \T \rightarrow \Q$ with $\N$ an $\A$-acyclic crossed module, it is not true in general that all morphisms from $\A$ to $\T$ factor through $\Q$.

By definition of $p^\Q$ we have the following equality for the composition $p^\Q \circ \varphi = 1 = h \circ 1$ as below. Therefore, any morphism from $\A$ to $\Q$ induces one from $\A$ to $\W$: 
 \begin{equation}\label{pullbackphipsi}
 \begin{tikzpicture}[baseline=(B.base),descr/.style={fill=white},scale=0.8]
\node (A) at (0,0) {$\PAT$};
\node (B) at (0,2) {$\W$};
\node (C) at (2,2) {$\Q$};
\node (D) at (2,0) {$\PAQ$};
\node (E) at (-1.75,3.75) {$A$};
  \path[-stealth,font=\scriptsize]
  (A.east) edge node[below] {$h$} (D.west) 
 (C.south) edge node[right] {$p^\Q$} (D.north) 
 (B.south) edge node[right] {$f$} (A.north) 
 (B.east)  edge node[above] {$g$} (C.west);
\draw[-stealth,font=\scriptsize,>=latex] (E.south) to[bend right=15]node[left] {$1$} (A.west);
\draw[-stealth,font=\scriptsize,>=latex] (E.east) to[bend left=15]node[right] {$ \; \varphi $} (C.north);
 \path[-stealth,font=\scriptsize,dashed,font=\scriptsize]
 (E.south east)  edge node[descr] {$\exists ! \psi $} (B.north);
\end{tikzpicture}
 \end{equation}
 We call $\psi$ the morphism determined by $\varphi$ and it makes sense now to consider $\K_\W$, the normal closure in $\W$ of the image of $\M \rightarrow \W$.
\end{remark}
 
 \begin{lemma} \label{isokernel}
 With the same notation as in \cref{rem:ideabehindnextproof}, we have an isomorphism $\K_\W \cong \K_\Q$.
 \end{lemma} 

\begin{proof}
 Limits are computed levelwise for crossed modules, so the pullback $\W$ consists of compatible pairs $(x, q)$ for $x \in (\sf P_A T)_i$ and $q \in \sf Q_i$ for $i = 1, 2$. By construction of $\psi$ we have $\psi(a) = (1, \varphi(a))$.


Now, we compute the kernels of the cokernels of $ev \colon \M \to \Q$ and $(1,ev) \colon \M \to \W$.
We have the two following descriptions of the kernels.
\[ \K_\Q  = \Big( ev_1(M_1)_{Q_2} [ev_2(M_2)_{Q_2}, Q_1 ],\; ev_2(M_2)_{Q_2}, \partial \Big) \]
\[ \K_\W =  \Big( (1,ev_1)(M_1)_{W_2} [(1,ev_2)(M_2)_{W_2}, W_1 ] , \; (1,ev_2)(M_2)_{W_2}, \partial' \Big) \]

The second group of the crossed module $\K_\W $ is the easier one:
\begin{align*}
    (1,ev_2)(M_2)_{W_2} &= \{ ^{(t_2,q_2)}(1, ev_2(m_2)) \mid (t_2,q_2) \in W_2, \; m_2 \in M_2 \}\\
    &= \{ (1,\;^{q_2} ev_2(m_2)) \mid q_2 \in Q_2, \; m_2 \in M_2 \} \\
    &= 1 \times ev_2(M_2)_{Q_2}.
\end{align*}
where the second equality holds since $h$ is surjective.
Via similar computations, we see that $ (1,ev_1)(M_1)_{W_1} = 1 \times ev_1(M_1)_{Q_1}$, so we are left with proving that 
\[ 
[(1,ev_2)(M_2)_{W_2}, W_1 ] = 1 \times [ev_2(M_2)_{Q_2}, Q_1 ] 
\]
This we do via the following equalities:
\begin{align*}
    [(1,ev_2)(M_2)_{_2}, W_1 ] &=  [(1 \times ev_2(M_2)_{Q_2}), W_1 ] \\
    &= \{ ^{(1,x_2)}(t_1,q_1)(t_1,q_1)^{-1} \mid x_2  \in ev_2(M_2)_{Q_2}, \; (t_1,q_1) \in W_1\}\\
    &= \{ (1,^{x_2}q_1q_1^{-1} \mid x_2  \in ev_2(M_2)_{Q_2}, \; q_1 \in Q_1\}\\
    &= 1 \times [(ev_2(M_2)_{Q_2}, Q_1 ]
\end{align*}
So finally we can conclude that $K_\W = 1 \times K_\Q$, in particular $K_\W$ and $K_\Q$ are isomorphic.
\end{proof}

\begin{prop}
\label{prop:induction}
For any ordinal $\beta$, we have a commutative diagram 
 \[ 
 \begin{tikzpicture}[rotate=90,descr/.style={fill=white},baseline=(A.base), scale=1]
\node (A) at (2,0) {$\PAT$};
\node (B) at (2,2) {$\W_\beta$};
\node (C) at (0,2) {$\Q_\beta$};
\node (D) at (0,0) {$\PAQ$};
\node (E) at (2,4) {$\W$};
\node (F) at (0,4) {$\Q$};
\node (X) at (1,1) {$(2)$};
  \path[-stealth,font=\scriptsize]
 (C) edge node[above] {$ $} (D) 
 (B) edge node[above] {$ $} (A) 
 (B)  edge node[left] {$ h_{\beta}$} (C)
 (E)  edge node[left] {$ g $} (F)
 (E) edge node[above] {$ f_{\beta} $} (B)
 (F) edge node[above] {$ p^\Q_{\beta} $} (C) 
 (A)  edge node[right] {$h $} (D);
\end{tikzpicture}
\] 
where $(2)$ is a pullback square, the maps $f_{\beta} \colon \W \to \W_\beta $ and $p^\Q_{\beta} \colon \Q \to \Q_\beta$ are $\P_A$-equivalences, and $h_{\beta}$ is a regular epimorphism.
\end{prop}

\begin{proof}
We prove it by induction. Since the nullification uses possibly a transfinite construction we have to initialize the induction, but the case $\beta =0$ holds by assumption, and then check the statement for successor and limit ordinals.

\underline{The successor case}
Suppose that for an ordinal $\beta$ the lemma is proved. 
Then we consider the kernels $\K^\W _{\beta}$ and $\K^\Q_{\beta}$ of the cokernels of the evaluation maps $ev: \coprod_{Hom(\A,\Q_{\beta})}\A \longrightarrow \Q_{\beta}$ and $ ev: \coprod_{Hom(\A,\Q_{\beta})}\A \longrightarrow \W _{\beta} $ respectively. They fit in the following diagram of exact rows:
 \[ \begin{tikzpicture}[rotate=90,descr/.style={fill=white},baseline=(A.base)]
\node (A) at (2,0) {$\W _{\beta+1}$};
\node (B) at (2,2) {$\W _\beta$};
\node (C) at (0,2) {$\Q_\beta$};
\node (D) at (0,0) {$\Q_{\beta+1}$};
\node (E) at (2,4) {$\K^\W _\beta$};
\node (F) at (0,4) {$\K^\Q_\beta$};
\node (X) at (1,1) {$(2)$};
  \path[-stealth,font=\scriptsize]
 (C) edge node[above] {$ p^\Q_{(\beta \to \beta + 1)}$} (D) 
 (B) edge node[above] {$f_{(\beta \to \beta + 1)} $} (A) 
 (B)  edge node[right] {$h_{\beta}$} (C)
 (E)  edge node[right] {$ \cong $} (F)
 (E) edge node[above] {$i_\W$} (B)
 (F) edge node[above] {$i_\Q$} (C); \path[-stealth,font=\scriptsize,dashed,font=\scriptsize]
 (A)  edge node[right] {$\exists ! h_{\beta+1} $} (D);
\end{tikzpicture}
\] 
Lemma \ref{isokernel} applies here and gives us the isomorphism between $\K^\W _\beta$ and $\K^\Q_\beta$. The composition
\[
p^\Q_{(\beta \to \beta + 1)} \circ h_{\beta} \circ i_\W  \colon \K_\beta^\W   \rightarrow \Q_{\beta+1}
\]
is zero by commutativity, yielding by the universal property of the cokernel the morphism $h_{\beta+1} \colon \W _{\beta+1} \rightarrow \Q_{\beta+1}$. The isomorphism between the kernels implies that $(2)$ is a pullback (see \cref{usefulproposemiab}). By induction hypothesis $h_{\beta}$ is a regular epimorphism and the composition $p^\Q_{(\beta \to \beta + 1)} \circ h_{\beta} \colon \W _\beta \rightarrow \Q_{\beta+1}$ is also a regular epimorphism, hence so is $h_{\beta +1}$. We show now that $p^\Q_{(\beta \to \beta + 1)}$ and $f_{(\beta \to \beta + 1)}$ are $\P_\A$-equivalences.

For the first one we write the cokernel $\Q_{\beta+1}$ as the pushout along the evaluation morphism:
\[ 
\begin{tikzpicture}[descr/.style={fill=white},baseline=(A.base)]
\node (A) at (0,0) {$1$};
\node (B) at (0,2) {$\coprod \A$};
\node (C) at (2,2) {$\Q_\beta$};
\node (D) at (2,0) {$\Q_{\beta+1}$};
  \path[-stealth,font=\scriptsize]
 (C.south) edge node[right] {$p^\Q_{(\beta \to \beta + 1)}$} (D.north) 
 (B.south) edge node[left] {$1$} (A.north) 
 (B.east) edge node[above] {$\varphi$} (C.west) 
 (A.east) edge node[below] {$inc$} (D.west) 
 ; 
\end{tikzpicture}\] 
where the coproduct is taken over $Hom(\A,\Q)$. 
The trivial map $\A \rightarrow 1$ is a $\P_\A$-equivalence, thus so is the pushout $p^\Q_{(\beta \to \beta + 1)
} \colon \Q_\beta \to \Q_{\beta+1}$ by \cref{lemma_colimitandequivalences} (1). 
By composing with the $\P_\A$-equivalence $\Q \to \Q_\beta$ we see that $p^\Q_{\beta + 1} \colon Q \to Q_{\beta +1}$ is a $\P_\A$-equivalence as well. The same argument shows that $f_{\beta +1} : \W  \rightarrow \W _{\beta + 1} $ is also a $\P_\A$-equivalence. By the universal property of the localization, we obtain two maps, one from $\W _{\beta + 1}$ to $\PAT$ and the other from $\Q_{\beta + 1}$ to $\PAQ$ such that $(2)$ commutes:
\[ 
\begin{tikzpicture}[rotate=90,descr/.style={fill=white},baseline=(A.base),yscale=1.4]
\node (A) at (2,0) {$\PAT$};
\node (B) at (2,2) {$\W _{\beta + 1}$};
\node (C) at (0,2) {$\Q_{\beta +1}$};
\node (D) at (0,0) {$\PAQ$};
\node (E) at (2,3.6) {$\W $};
\node (F) at (0,3.6) {$\Q$};
\node (X) at (1,1) {$(2)$};
\node (Y) at (1,3) {$(1)$};
  \path[-stealth,font=\scriptsize]
 (B)  edge node[left] {$h_{\beta + 1}$} (C)
 (E)  edge node[right] {$ g $} (F)
 (E) edge node[above] {$f_{\beta + 1}$} (B)
 (F) edge node[above] {$p^\Q_{\beta + 1}$} (C)
 (A)  edge node[right] {$h$} (D);
 \draw[-stealth,font=\scriptsize,>=latex] (E) to[bend left=40]node[descr] {$ f $} (A);
 \draw[-stealth,font=\scriptsize,>=latex] (F) to[bend right=40]node[descr] {$ p^\Q $} (D);
  \path[-stealth,font=\scriptsize,dashed,font=\scriptsize]
 (B)  edge node[above] {$ $} (A);
  \path[-stealth,font=\scriptsize,dashed,font=\scriptsize]
 (C)  edge node[above] {$ $} (D);
\end{tikzpicture}
\]
Since $(1)$ and the outer rectangle are pullbacks and $h_{\beta + 1}$ is a regular epimorphism, we can conclude by Proposition 4.1.4 in \cite{BB} that $(2)$ is a pullback.\\ \\
\underline{The limit case }
To prove the statement for a general transfinite induction we need to prove it for a limit ordinal as well. Let $\gamma$ be a limit ordinal and
\[
\Q_{\gamma} = \hbox{\rm colim}_{\alpha < \gamma} \Q_{\alpha} \ \ \ \ \W _{\gamma} = \hbox{\rm colim}_{\alpha < \gamma} \W _{\alpha}
\]
We have shown that $p^\Q_{(\alpha-1 \rightarrow \alpha)} \colon \Q_{\alpha-1} \rightarrow \Q_{\alpha}$ is a $\P_\A$-equivalence for all $\alpha < \gamma$. Hence the composition $p^\Q_{\alpha} \colon \Q \rightarrow \Q_{\alpha}$ is also a $\P_\A$-equivalence and \cref{lemma_colimitandequivalences} (3), implies that $p^\Q_{\gamma} \colon \Q \rightarrow \Q_{\gamma}$ is a $\P_\A$-equivalence. The same reasoning holds for $f_{\gamma} \colon \W  \rightarrow \W _{\gamma}$. The existence of the maps $f: \W \rightarrow \PAT$ and $p^\Q : \Q \rightarrow \PAQ$ give us two maps $\W_ \gamma \rightarrow \PAT$ and $\Q_\gamma \rightarrow \PAQ$ as shown on the diagram below \eqref{diagordinallimite}.

 The nullification $\sf P_A Q$ is constructed as filtered colimit of the $\Q_\alpha$, see \cref{propnullify}. Filtered colimits commutes with finite limits, in particular with kernels. 
 Therefore
\[
\K^\Q_\gamma := \ker(\Q \rightarrow \Q_{\gamma}) \cong \hbox{\rm colim}_{\alpha < \gamma}\ker(\Q \rightarrow \Q_{\alpha})
\]
where $\ker(\Q \rightarrow \Q_\alpha)$ will be denoted $\K^\Q_\alpha$. The category $\sf XMod$ is a variety of algebras (also called algebra category of fixed type). Hence, by \cite[Proposition~IX.1.2]{MacLane}, we know that the forgetful functor $ \mathcal{U} \colon \sf XMod \rightarrow Set$ creates filtered colimits. In other words we have : 
\[ 
\mathcal{U}(\hbox{\rm colim}_{\alpha < \gamma} \K^\Q_{\alpha}) = \hbox{\rm colim}_{\alpha < \gamma} \mathcal{U} \K^\Q_\alpha = \bigcup_{\alpha < \gamma} \mathcal{U} \K^\Q_\alpha 
\] 
where the colimit in the first term lies in the category of crossed modules and the second colimit in the category of sets. This means that we know the structure of $\hbox{\rm colim}_{\alpha < \gamma}\K^\Q_\alpha$ as a set. Now since $\K^\Q_\alpha \cong \K^\W_\alpha$ for all $\alpha < \gamma$  and $\K^\Q_\gamma$ can be written as a union of $\K^\Q_\alpha$ (as well as $\K^\W_\gamma$) we conclude that $\K_\gamma^\Q \cong \K_\gamma^\W$. We consider now the diagram:

\begin{equation}\label{diagordinallimite} \begin{tikzpicture}[rotate=90,descr/.style={fill=white},baseline=(X.base),yscale=1.4]
\node (A) at (2,0) {$\PAT$};
\node (B) at (2,2) {$\W_\gamma$};
\node (C) at (0,2) {$\Q_\gamma$};
\node (D) at (0,0) {$\PAQ$};
\node (E) at (2,3.6) {$\W $};
\node (F) at (0,3.6) {$\Q$};
\node (X) at (1,1) {$(1)$};
\node (Y) at (1,3) {$(2)$};
  \path[-stealth,font=\scriptsize]
 (B)  edge node[left] {$h_{\gamma}$} (C)
 (E)  edge node[right] {$ g $} (F)
 (E) edge node[above] {$f_{\gamma}$} (B)
 (F) edge node[above] {$p^\Q_{\gamma}$} (C)
 (A)  edge node[right] {$h$} (D);
 \draw[-stealth,font=\scriptsize,>=latex] (E) to[bend left=40]node[descr] {$ f $} (A);
 \draw[-stealth,font=\scriptsize,>=latex] (F) to[bend right=40]node[descr] {$ p^\Q $} (D);
  \path[-stealth,font=\scriptsize,dashed,font=\scriptsize]
 (B)  edge node[above] {$ $} (A);
  \path[-stealth,font=\scriptsize,dashed,font=\scriptsize]
 (C)  edge node[above] {$ $} (D);
\end{tikzpicture}
\end{equation}
Since the kernels of $f_{\gamma}$ and $p^\Q_{\gamma}$ are isomorphic we deduce that $(2)$ is a pullback. As we have shown that every map $p^\Q_{(\alpha \rightarrow \alpha+1)} \colon \Q_\alpha \rightarrow \Q_{\alpha+1}$ is a regular epimorphism, the morphism $p^\Q_{\alpha} \colon \Q \rightarrow \Q_\alpha$ is also a regular epimorphism, being a composition of regular epimorphisms in a regular category. 
The colimit functor 
being a left adjoint functor, it preserves colimits and in particular cokernels. In a pointed protomodular category, any regular epimorphism is a cokernel, therefore
\[
p^\Q_{\gamma} \colon \Q \rightarrow \Q_\gamma
\]
is a regular epimorphism. 
The composition $p^\Q_{\gamma} \circ g$ is also a regular epimorphism, and we conclude that so is $h_{\gamma}$. With the same argument as for the successor step, we get that $(1)$ is a pullback, which ends the induction proof. 
\end{proof}

 We are ready now for the main result of this section.
\begin{theorem}
\label{thm:PAadmissible}
Let $\A$ be any crossed module. The nullification functor $\P_\A$ is admissible for the class of regular epimorphisms.
\end{theorem}

\begin{proof}
 Let $\W$ be the pullback of a regular epimorphism $h\colon \PAT \to \PAQ$ between $\P_\A$-local crossed modules along the localization morphism $p^\Q\colon \Q \to \PAQ$.
Let $\lambda$ be the ordinal such that $\Q_\lambda \cong \PAQ$ (see \cref{propnullify}). By Proposition~\ref{prop:induction} we have a diagram: 
\[ 
\begin{tikzpicture}[rotate=90,descr/.style={fill=white},baseline=(A.base)]
\node (A) at (2,0) {$\PAT$};
\node (B) at (2,2) {$\W_\lambda$};
\node (C) at (0,2) {$\Q_\lambda$};
\node (D) at (0,0) {$\PAQ$};
\node (E) at (2,4) {$\W$};
\node (F) at (0,4) {$\Q$};
\node (X) at (1,1) {$(2)$};
  \path[-stealth,font=\scriptsize]
 (C) edge node[above] {$ \cong $} (D) 
 (B) edge node[above] {$ $} (A) 
 (B) edge node[left] {$h_{\lambda}$} (C) 
 (E) edge node[left] {$g$} (F)
 (A) edge node[right] {$h$} (D) 
 (E) edge node[above] {$f_{\lambda}$} (B)
 (F) edge node[above] {$ p_{\lambda}^\Q$} (C) 
 ; 
\end{tikzpicture}
\]
 where the outer rectangle is a pullback, the morphisms $f_{\lambda}$ and $p^\Q_{\lambda}$ are $\P_\A$-equivalences, and $(2)$ is a pullback. Since isomorphisms are stable under pullbacks, we have an isomorphism $\W_\lambda \cong \PAT$. We have thus proved that the map $f\colon \W \to \PAT$ is a $\P_\A$-equivalence, which means that
the functor $\P_\A$ is admissible.
\end{proof}

In this article we have focused on regular-epi localization functors because they appear naturally when studying conditional flatness and admissibility in the category of groups, crossed modules, or more general semi-abelian categories.
We conclude this section by observing that the
notion of conditional flatness can also be defined for non regular-epi localization functor. The next proposition gives an example of such a localization functor which is conditionally flat. Let us stress that we will not a priori have an equivalence with admissiblity, as was the case for regular-epi localization functors by \cref{finalstep}. In the proof of the following proposition we have thus to verify the more general condition for conditional flatness, as in \cref{defcondflat}.


\begin{prop}
\label{condflatnotregularepi}
    There exists a non regular-epi localization functor which is nevertheless conditionally flat and therefore admissible for the class of regular epimorphisms.
\end{prop}

\begin{proof}
    We consider the functor $\I$ defined in \cref{I} which sends any crossed module $(N_1,N_2,\partial^\N)$ to $(N_2,N_2,Id_{N_2})$. This functor is not regular-epi because if we consider a crossed module for which the connecting morphism is not surjective then the localization morphism will not be a regular epimorphism.  

    We prove now that $\I$ is conditional flat. Let 
\begin{tikzpicture}[descr/.style={fill=white},scale= 0.8, baseline=(A.base)] 
\node (A) at (0,0) {$\T$};
\node (B) at (2.5,0) {$\Q$};
\node (C) at (-2.5,0) {$\N$};
\node (O1) at (-4.5,0) {$1$};
\node (O2) at (4.5,0) {$1$};
\path[-stealth,font=\scriptsize]
(C.east) edge node[above] {$\kappa$} (A.west)
(O1) edge node[above] {$ $} (C)
(B) edge node[above] {$ $} (O2)
(A) edge node[above] {$\alpha$} (B);
\end{tikzpicture}
    be any exact sequence of crossed modules. We see that $\I((N_1,N_2,\partial^\N)) = (N_2,N_2,Id_{N_2})$ is a normal subcrossed module of $ (T_2,T_2,Id_{T_2}) = \I((T_1,T_2,\partial^\T)$ and that $\I((Q_1,Q_2,\partial^\Q))= (Q_2,Q_2,Id_{Q_2})$ is the cokernel of $\kappa \colon \N \to \T$. Therefore any exact sequence of crossed modules is $\I$-flat. In particular any pullback along any morphism of crossed modules of an $\I$-flat exact sequence is $\I$-flat, hence $\I$ is conditionally flat.
\end{proof}

\printbibliography
\end{document}